\documentclass{amsart}

\usepackage[utf8]{inputenc}
\usepackage[T1]{fontenc}
\usepackage{amsmath}
\usepackage{amssymb}
\usepackage{amsfonts}
\usepackage{amsbsy}
\usepackage{euscript}
\usepackage{array}
\usepackage{graphicx}
\usepackage{color}
\usepackage{float}
\usepackage[normalem]{ulem}
\usepackage{booktabs}
\usepackage[hidelinks]{hyperref}
\hypersetup{
  pdftitle={A general framework for the OSRC-preconditioned EFIE in computational electromagnetics},
  pdfauthor={Ignacia Fierro-Piccardo}
}

\newtheorem{Lem}{Lemma}  
 
\newtheorem{Cor}{Corollary}

\newcommand{\vn} {{\pmb{\nu}}}

\newcommand{\vu} {{\mathbf{u}}}

\newcommand{\curlg} {\mathop{\mathrm{curl}_{\Gamma}}\nolimits}
\newcommand{\vcurlg} {\mathop{\mathrm{\bf curl}_{\Gamma}}\nolimits}

\newcommand{\vnablag} {\nabla_{\Gamma}}

\newcommand{\divg} {\mathop{\mathrm{div}_{\Gamma}}\nolimits}
\newcommand{\Deltag} {\Delta_{\Gamma}}

\newcommand{\vDeltag} {{\bf \Delta}_{\Gamma}}

\newcommand{\vecv}[1]{\mathbf{#1}}

\newcommand{\vcurl}{\textbf{curl}}

\newcommand{\SLO}{\vecv{S}_{\kappa}}
\newcommand{\DLO}{\vecv{C}_{\kappa}}

\newcommand{\vx}{\mathbf{x}}
\newcommand{\vy}{\mathbf{y}}

\newcommand{\half}{\frac{1}{2}}

\newcommand{\curl}{\mathbf{curl}}

\newcommand{\sgrad}{\nabla_{\Gamma}}

\newcommand{\hsgrad}{\nabla_{\Gamma_{h}}}
\newcommand{\scurl}{\text{curl}_{\Gamma}}
\newcommand{\hscurl}{\text{curl}_{\Gamma_{h}}}
\newcommand{\vscurl}{\textbf{\text{curl}}_{\Gamma}}
\newcommand{\cnabla}[1]{\nabla_{#1}}

\newcommand{\magnetic}{\gamma_{\vecv{N}}}
\newcommand{\tangential}{\gamma_{\vecv{t}}}

\newcommand{\Htdiv}{\mathbf{H}^{-\half}_{\operatorname{div}}(\Gamma)}
\newcommand{\Htcurl}{\mathbf{H}^{-\half}_{\operatorname{curl}}(\Gamma)}

\newtheorem{proposition}{Proposition}

\title[OSRC-preconditioned EFIE]{A general framework for the OSRC-preconditioned EFIE in computational electromagnetics}
\author{Ignacia Fierro-Piccardo}
\date{}

\begin{document}

\begin{abstract}
We investigate on-surface radiation condition (OSRC) preconditioners
for the electric field integral equation (EFIE), considering
Electric-to-Magnetic (EtM) and Magnetic-to-Electric (MtE)
approximations applied from both the left and the right. A
pseudodifferential analysis establishes the well-posedness of the four
resulting formulations and shows that, locally, the OSRC EtM and MtE
approximations present a similar structure to the EFIE itself and its 
inverse, respectively. At the discrete level, both MtE formulations are 
compatible with the standard RWG/SNC discretization and can be implemented 
without barycentric refinement, whereas the EtM formulations require an appropriate
discrete duality map and BC/RBC spaces. Numerical experiments show that the OSRC 
preconditioners substantially improve GMRES convergence compared with the unpreconditioned 
EFIE, with the refinement-free MtE formulations also providing the lowest computational 
cost among the formulations considered.

\end{abstract}

\maketitle
\normalem

\section{Introduction}

The accurate and efficient numerical solution of wave scattering
problems remains an important challenge in computational
electromagnetics and acoustics. A fundamental feature of an exterior
scattering problem is that the scattered field must satisfy an outgoing
radiation condition at infinity. Since the exterior domain is unbounded,
a direct numerical treatment requires accounting for wave propagation
arbitrarily far from the scatterer. One approach is to introduce an
artificial boundary enclosing the scatterer and impose an approximate
radiation condition on this boundary. The unbounded problem is then
replaced by a problem posed on the bounded region between the scatterer
and the artificial boundary.

On-Surface Radiation Conditions (OSRCs) take this idea one step further
by imposing an approximate radiation condition directly on the surface
of the scatterer, thereby eliminating the need for an artificial
surrounding boundary. Introduced by Kriegsmann et al.
\cite{kriegsmann1987new} for two-dimensional electromagnetic scattering
by convex conducting bodies, the original OSRC approach was based on
asymptotic radiation conditions derived from the outgoing expansion of
the scattered field. By expressing these conditions in terms of normal
and tangential derivatives and local geometric quantities, they could be
applied directly on the surface of the scatterer, yielding an
inexpensive approximation of the scattered field. The approach was
subsequently developed for acoustic and electromagnetic wave propagation
\cite{jones1988approximate,jones1988surface,jones1992improved}, with
dedicated OSRC formulations later developed for Maxwell problems
\cite{roxburgh1997electromagnetic,ammari1998surface}.

From the perspective of boundary operators, an OSRC can be interpreted
as an approximation of the exact map relating the boundary traces of an
outgoing exterior field. Such a map is nonlocal: its value at a given point on the boundary
generally depends on the boundary data over the entire surface. 
For the scalar Helmholtz equation, these
operators correspond to the Dirichlet-to-Neumann (DtN) and
Neumann-to-Dirichlet (NtD) maps. For Maxwell's equations, the analogous
operators are the Electric-to-Magnetic (EtM) and Magnetic-to-Electric
(MtE) maps, which relate the tangential electric and magnetic traces.
OSRCs provide computationally tractable approximations of these maps in
terms of differential or pseudodifferential operators acting on the
scattering surface.

Later developments based on microlocal analysis provided a systematic
framework for the construction of radiation conditions of higher order
\cite{antoine2001microlocal}. This operator perspective was subsequently
exploited in the design of analytical preconditioners for regularized
integral formulations for acoustic and electromagnetic scattering
\cite{antoine2004analytic,antoine2007generalized}. Closely related
pseudodifferential regularization strategies have also been developed
for well conditioned electromagnetic integral equations
\cite{boubendir2014well}. Broader accounts of OSRC methods and operator
preconditioning for wave scattering can be found in
\cite{antoine2008advances,antoine2021introduction}.

More recently, OSRC operators have been incorporated directly into
boundary element discretizations. Van't Wout et al.
\cite{van2022frequency} used OSRC operators to precondition boundary
integral formulations for acoustic transmission problems. For
electromagnetic scattering, Fierro-Piccardo and Betcke
\cite{BetckeFierroPiccardo} used the finite element representation
proposed by El Bouajaji et al. \cite{el2014approximate} to construct an
approximate Magnetic-to-Electric (MtE) map for the Electric Field
Integral Equation (EFIE). The resulting preconditioner provides an
effective refinement-free alternative to standard Calder\'on
implementations based on Buffa--Christiansen spaces.

The present work provides continuous and discrete analyses of OSRC
preconditioning for the EFIE, treating both EtM and MtE approximations
under left and right composition. At the continuous level, we show
that, locally, the MtE and EtM formulations can be interpreted, respectively,
as approximate inverse and Calder\'on type preconditioning strategies.
The pseudodifferential structure underlying the EtM regularization is
closely related to that appearing in earlier analyses of regularized
combined-field formulations
\cite{boubendir2014well}. Building on this
approach, we place the EtM and MtE formulations within a common
framework and analyze both their left and right compositions with the
EFIE. We identify their principal parts, keeping explicit the role of
the tangential rotations, and show that each formulation is a compact
perturbation of a boundedly invertible operator and hence has a second
kind Fredholm structure.

A key contribution of this work is to characterize how these continuous
formulations translate to a Galerkin discretization. Although all four
formulations have a similar second kind structure at the continuous
level, we show that their discretizations have fundamentally
different duality requirements. In particular, the MtE formulations are
naturally compatible with the RWG/SNC pairing, allowing refinement-free
implementations for both left and right preconditioning. In contrast,
the EtM formulations require additional discrete duality maps, which
involve Buffa--Christiansen spaces on a
barycentrically refined mesh, as in the standard Calder\'on
multiplicative preconditioner (CMP)~\cite{andriulli2008multiplicative}.
This analysis explains why the refinement-free construction available
for MtE preconditioning does not extend directly to the corresponding
EtM formulations.

Finally, numerical experiments compare the spectral and iterative behaviour of 
the different formulations and quantify the computational impact of their discretizations. 
The results support our continuous analysis, 
demonstrate substantial reductions in GMRES iterations, and show that the refinement-free 
MtE formulations provide a particularly inexpensive constructions of OSRC preconditioning. 
Additional experiments on an open surface indicate that the effectiveness of the preconditioners 
can extend beyond the class of geometries covered by us in the continuous setting.

\section{Electromagnetic scattering, functional setting, and boundary operators}
\label{sect:formulations}

Let $\Omega^-\subset\mathbb R^3$ be a bounded domain with smooth
boundary $\Gamma=\partial\Omega^-$, and let
$\Omega^+=\mathbb R^3\setminus\overline{\Omega^-}$. We consider the
scattering of a time-harmonic incident electromagnetic field
$(\vecv e^i,\vecv h^i)$ by a perfectly electrically conducting obstacle
occupying $\Omega^-$. The total field is
\[
(\vecv e^t,\vecv h^t)
=
(\vecv e^i+\vecv e^s,\vecv h^i+\vecv h^s),
\]
where the scattered field satisfies
\begin{equation}
\label{Maxwell}
\begin{cases}
\curl\,\vecv h^s+i\kappa\vecv e^s=0,
& \text{in }\Omega^+,\\
\curl\,\vecv e^s-i\kappa\vecv h^s=0,
& \text{in }\Omega^+,
\end{cases}
\end{equation}
with $\kappa=\omega\sqrt{\varepsilon_0\mu_0}$. The PEC boundary
condition is
\begin{equation}
\label{eq:conducting}
\vecv e^t_{|\Gamma}\times\vn=0,
\end{equation}
and the scattered field satisfies the Silver--Müller radiation
condition
\begin{equation}
\label{Silver-Muller}
\lim_{|\vx|\to\infty}
|\vx|
\left(
\curl\,\vecv h^s\times\frac{\vx}{|\vx|}
+\vecv e^s
\right)
=0.
\end{equation}

We next introduce the functional setting and surface differential
operators used throughout this text: the surface gradient, vector curl,
divergence, and scalar curl are denoted by
$\vnablag$, $\vcurlg$, $\divg$, and $\curlg$, respectively. They satisfy
\[
\curlg\vnablag u=0,
\qquad
\divg\vcurlg u=0,
\]
and
\[
\divg(\vn\times\vecv v)=-\curlg\vecv v,
\qquad
\curlg(\vn\times\vecv v)=\divg\vecv v.
\]
The scalar Laplace--Beltrami operator and vector Hodge Laplacian are
defined by
\begin{equation}
\label{eq:lap_belt}
\Deltag u
=
\divg\vnablag u
=
-\curlg\vcurlg u,
\end{equation}
and
\begin{equation}
\label{eq:hodge_lap}
\vDeltag\vecv v
=
\vnablag\divg\vecv v
-
\vcurlg\curlg\vecv v,
\end{equation}
respectively.

We denote by $\mathrm H^s(\Gamma)$ the usual scalar Sobolev spaces on
$\Gamma$, with $\mathrm H^0(\Gamma)=\mathrm L^2(\Gamma)$, and by
$\mathbf H_\times^s(\Gamma)$ the corresponding spaces of tangential
vector fields. For a domain $D\subset\mathbb R^3$, we define
\begin{align*}
\mathbf H(\curl,D)
&=
\left\{
\vecv v\in\vecv L^2(D):
\curl\,\vecv v\in\vecv L^2(D)
\right\},\\
\mathbf H(\curl^2,D)
&=
\left\{
\vecv v\in\mathbf H(\curl,D):
\curl\,\curl\,\vecv v\in\vecv L^2(D)
\right\}.
\end{align*}
And we define the relevant tangential vector spaces
\[
\Htdiv
=
\left\{
\vecv v\in\mathbf H_\times^{-\half}(\Gamma):
\divg\vecv v\in \mathrm H^{-\half}(\Gamma)
\right\},
\]
and
\[
\Htcurl
=
\left\{
\vecv v\in\mathbf H_\times^{-\half}(\Gamma):
\curlg\vecv v\in \mathrm H^{-\half}(\Gamma)
\right\}.
\]

The spaces $\Htdiv$ and $\Htcurl$ are related by the tangential
rotation
\[
\mathcal J\vecv v=\vn\times\vecv v.
\]
Indeed, the surface differential identities above imply
\[
\mathcal J:\Htdiv\longrightarrow\Htcurl,
\qquad
\mathcal J:\Htcurl\longrightarrow\Htdiv,
\qquad
\mathcal J^2=-I.
\]
Moreover, $\Htdiv$ is self-dual with respect to the twisted duality
pairing
\begin{equation}
\label{eq:twisted_inner_product}
\langle\vecv j,\vecv m\rangle_\Gamma
=
(\vecv j,\mathcal J\vecv m)_\Gamma
=
\int_\Gamma
(\vn\times\vecv m)\cdot\vecv j\,d\Gamma,
\qquad
\vecv j,\vecv m\in\Htdiv.
\end{equation}
For sufficiently regular tangential fields, the pairing is given by
the surface integral above. It extends continuously to
$\Htdiv\times\Htdiv$ \cite[Theorem~2]{Buffa2003}.

We now pass from the volume Maxwell problem to its boundary integral
representation. For a sufficiently smooth vector field $\vecv v$, we
define the tangential and magnetic traces by
\[
\tangential(\vecv v)=\vecv v\times\vn,
\qquad
\magnetic(\vecv v)
=
(i\kappa)^{-1}\tangential(\curl\vecv v),
\]
with superscripts $+$ and $-$ denoting exterior and interior traces.
For a generic trace $\gamma$, we set
\[
[\gamma]_\Gamma=\gamma^+-\gamma^-,
\qquad
\{\gamma\}_\Gamma=\frac{\gamma^++\gamma^-}{2}.
\]
The tangential trace defines the continuous surjection
\[
\tangential:\mathbf H(\curl,\Omega)\longrightarrow\Htdiv
\]
\cite[Theorem~1]{Buffa2003}. Consequently, if
$\vecv v\in\mathbf H(\curl^2,\Omega)$, then
$\curl\vecv v\in\mathbf H(\curl,\Omega)$ and both
$\tangential(\vecv v)$ and $\magnetic(\vecv v)$ are well defined as
elements of $\Htdiv$.

The scattered electromagnetic fields admit the Stratton--Chu
representation
\cite{colton2013integral,nedelec2001acoustic}
\begin{equation}
\label{stratton-chu}
\begin{aligned}
\vecv e^{s}(\vx)
&=
-\mathcal{T}\bigl(\tangential \vecv e^s\bigr)(\vx)
-\mathcal{K}\bigl(\magnetic \vecv e^s\bigr)(\vx),
\\
\vecv h^{s}(\vx)
&=
\phantom{-}\mathcal{K}\bigl(\tangential \vecv e^s\bigr)(\vx)
-\mathcal{T}\bigl(\magnetic \vecv e^s\bigr)(\vx),
\qquad \vx\in\Omega^+.
\end{aligned}
\end{equation}
For a sufficiently smooth tangential density $\vecv u$, the electric
and magnetic potentials are given, for $\vx\notin\Gamma$, by
\begin{align}
\mathcal{T}\vecv u(\vx)
&=
\frac{i}{\kappa}
\curl_{\vx}\curl_{\vx}
\int_{\Gamma}
\mathrm G_{\kappa}(\vx,\vy)\vecv u(\vy)\,d\Gamma(\vy)
\nonumber\\
&=
i\kappa
\int_{\Gamma}
\mathrm G_{\kappa}(\vx,\vy)\vecv u(\vy)\,d\Gamma(\vy)
-
\frac{1}{i\kappa}\nabla_{\vx}
\int_{\Gamma}
\mathrm G_{\kappa}(\vx,\vy)
\divg\vecv u(\vy)\,d\Gamma(\vy),
\label{eq:electric_potential}
\\
\mathcal{K}\vecv u(\vx)
&=
\curl_{\vx}
\int_{\Gamma}
\mathrm G_{\kappa}(\vx,\vy)\vecv u(\vy)\,d\Gamma(\vy),
\label{eq:magnetic_potential}
\end{align}
where
\begin{equation}
\label{eq:helmholtz_green}
\mathrm G_{\kappa}(\vx,\vy)
=
\frac{\exp(i\kappa|\vx-\vy|)}
{4\pi|\vx-\vy|},
\qquad
\vx\neq\vy,
\end{equation}
is the fundamental solution of the three-dimensional Helmholtz
equation.

Taking exterior tangential traces of these potentials gives the
electric and magnetic boundary integral operators

\[
\SLO,\DLO:\Htdiv\longrightarrow\Htdiv,
\]
with the jump relations
\begin{equation}
\label{eq:ids_2}
\tangential^{+}\mathcal{T}
=
\SLO,
\qquad
\tangential^{+}\mathcal{K}
=
-\frac{\vecv I}{2}+\DLO.
\end{equation}
In particular, the EFIE operator can be written as \cite[equation 31]{Buffa2003}
\begin{equation}
\label{eq:efie2}
\SLO
=
\mathcal J
\left(
-i\kappa\pmb\Psi_\kappa
+
(i\kappa)^{-1}
\sgrad\circ\Psi_\kappa\circ\divg
\right)
:
\Htdiv\longrightarrow\Htdiv,
\end{equation}
where $\Psi_\kappa$ is the scalar Helmholtz single-layer boundary
operator
\begin{equation}
\label{eq:sl_op}
\Psi_\kappa(\phi)(\vx)
=
\gamma^+
\int_\Gamma
\mathrm G_\kappa(\vx,\vy)\phi(\vy)\,d\Gamma(\vy),
\end{equation}
and $\pmb\Psi_\kappa$ denotes its componentwise extension to
tangential vector fields.

Taking both traces of the representation formula
\eqref{stratton-chu} yields the exterior Calder\'on projector
\begin{equation}
\label{eq:calderon_projector}
\underbrace{\begin{bmatrix}
\dfrac{\vecv I}{2}-\DLO & -\SLO\\[1ex]
\SLO & \dfrac{\vecv I}{2}-\DLO
\end{bmatrix}}_{\mathcal C}
\begin{bmatrix}
\tangential^{+}\vecv e^{s}\\
\magnetic^{+}\vecv e^{s}
\end{bmatrix}
=
\begin{bmatrix}
\tangential^{+}\vecv e^{s}\\
\magnetic^{+}\vecv e^{s}
\end{bmatrix}.
\end{equation}
Since $\mathcal C^2=\mathcal C$, the operators satisfy the Calder\'on
identities
\begin{equation}
\label{eq:calderon_identities}
\DLO\SLO=-\SLO\DLO,
\qquad
\SLO^2
=
\DLO^2-\frac{\vecv I}{4}
=
-
\left(
\frac{\vecv I}{2}-\DLO
\right)
\left(
\frac{\vecv I}{2}+\DLO
\right).
\end{equation}
In particular,
\begin{equation}
\label{eq:calderon_intertwining}
\left(
\frac{\vecv I}{2}+\DLO
\right)\SLO
=
\SLO
\left(
\frac{\vecv I}{2}-\DLO
\right),
\qquad
\SLO
\left(
\frac{\vecv I}{2}+\DLO
\right)
=
\left(
\frac{\vecv I}{2}-\DLO
\right)\SLO.
\end{equation}

For an outgoing solution of the exterior Maxwell problem, we define
the Electric-to-Magnetic (EtM) and Magnetic-to-Electric (MtE) maps by
\begin{equation}
\label{eq:EtM_MtE_def}
\pmb\Lambda:
\tangential^{+}\vecv e^{s}
\longmapsto
\magnetic^{+}\vecv e^{s},
\qquad
\vecv V:
\magnetic^{+}\vecv e^{s}
\longmapsto
\tangential^{+}\vecv e^{s}.
\end{equation}
The first row of \eqref{eq:calderon_projector} gives
\[
\SLO\magnetic^{+}\vecv e^{s}
=
-
\left(
\frac{\vecv I}{2}+\DLO
\right)
\tangential^{+}\vecv e^{s},
\]
and hence, formally,
\begin{equation}
\label{EtM}
\pmb\Lambda
=
-\SLO^{-1}
\left(
\frac{\vecv I}{2}+\DLO
\right).
\end{equation}
Whenever the corresponding inverses are well defined, the MtE map is
\begin{equation}
\label{MtE}
\vecv V
=
\pmb\Lambda^{-1}
=
-
\left(
\frac{\vecv I}{2}+\DLO
\right)^{-1}
\SLO.
\end{equation}

\begin{proposition}
The exact EtM and MtE maps provide two complementary preconditioning
strategies for the EFIE. Their left and right compositions with
$\SLO$ satisfy
\begin{align}
\SLO\pmb\Lambda
&=
-\left(
\frac{\vecv I}{2}+\DLO
\right),
&
\pmb\Lambda\SLO
&=
-\left(
\frac{\vecv I}{2}-\DLO
\right),
\\
\vecv V\SLO
&=
\left(
\frac{\vecv I}{2}-\DLO
\right),
&
\SLO\vecv V
&=
\left(
\frac{\vecv I}{2}+\DLO
\right).
\end{align}
\end{proposition}

\begin{proof}
Recall that
\[
\pmb\Lambda
=
-\SLO^{-1}
\left(
\frac{\vecv I}{2}+\DLO
\right),
\qquad
\vecv V
=
-
\left(
\frac{\vecv I}{2}+\DLO
\right)^{-1}
\SLO .
\]
Hence
\[
\SLO\pmb\Lambda
=
-\left(
\frac{\vecv I}{2}+\DLO
\right),
\qquad
\SLO\vecv V
=
\frac{\vecv I}{2}+\DLO .
\]

It follows that
\[
\left(
\frac{\vecv I}{2}+\DLO
\right)\SLO
=
\SLO
\left(
\frac{\vecv I}{2}-\DLO
\right).
\]
Therefore
\[
\pmb\Lambda\SLO
=
-\SLO^{-1}
\left(
\frac{\vecv I}{2}+\DLO
\right)\SLO
=
-\left(
\frac{\vecv I}{2}-\DLO
\right).
\]

Similarly, from
\[
\SLO
\left(
\frac{\vecv I}{2}+\DLO
\right)^{-1}
=
\left(
\frac{\vecv I}{2}-\DLO
\right)^{-1}
\SLO,
\]
together with the Calderón identity
\[
\SLO^2
=
\DLO^2-\frac14\vecv I,
\]
we obtain
\[
\vecv V\SLO
=
\frac{\vecv I}{2}-\DLO.
\]
\end{proof}

Thus, the exact EtM and MtE maps transform the EFIE into operators of
the second kind under both left and right composition. However, constructing 
$\pmb\Lambda$ or $\vecv V$ exactly is computationally
impractical, since their boundary integral representations involve the
inversion of dense operators. We therefore replace them by
approximations based on the on-surface radiation condition (OSRC)
approach
\cite{el2014approximate,FierroPiccardoPhD,BetckeFierroPiccardo}.
The following section introduces these approximations and investigates
whether the second kind structure of the exact compositions is
retained.

\section{OSRC-preconditioned EFIEs}\label{sect:theoretic_results}

\subsection{Approximation of the EtM and MtE operators}
We approximate the EtM map by \cite{darbas2004preconditionneurs, el2014approximate}:

\begin{equation}
\label{eq:EtM_1st_app}
\pmb\Lambda_{\varepsilon}
=
 \underbrace{
\left(
\vecv{I}+\frac{\vDeltag}{\kappa^2_{\varepsilon}}
\right)^{-\half}
}_{=:\pmb\Lambda_{1,\varepsilon}^{-1}}
\underbrace{
\left(
\vecv{I}-\frac{1}{\kappa_{\varepsilon}^2}\vcurlg\scurl
\right)
}_{=:\pmb\Lambda_{2,\varepsilon}}
\mathcal J: \Htdiv \mapsto \Htdiv.
\end{equation}

To avoid singularities, i.e. wavenumbers such that
$\vDeltag/\kappa^2\approx-\vecv I$, we choose
$\kappa_\varepsilon=\kappa+i\varepsilon$ with
$\varepsilon\neq0$.

Following \cite{el2014approximate}, the damping parameter is chosen as
\[
\varepsilon
=
0.39\,\kappa^{1/3}\mathcal H^{2/3}.
\]
In the present analysis, we replace the local mean curvature
$\mathcal H$ by a fixed reference value $\mathcal H_0>0$, and therefore $\kappa_\varepsilon=\kappa+i\varepsilon$ is constant on
$\Gamma$.

We define the OSRC approximation of the MtE operator by
\[
\vecv V_\varepsilon
:=
-\mathcal J \pmb\Lambda_{2,\varepsilon}^{-1}
\pmb\Lambda_{1,\varepsilon} :\Htdiv \mapsto \Htdiv.
\]

With these definitions, we wonder whether
\begin{align}
\pmb\Lambda_\varepsilon\SLO
&:
\Htdiv\rightarrow\Htdiv,
\label{eq:left-lambda}
\\
\vecv V_\varepsilon\SLO
&:
\Htdiv\rightarrow\Htdiv,
\label{eq:left-V}
\end{align}

define two left-preconditioned formulations of the EFIE operator.

A similar question can be posed for right
preconditioning of the EFIE:
\begin{align}
\SLO \pmb\Lambda_\varepsilon
&:
\Htdiv\rightarrow\Htdiv,
\label{eq:right-lambda}
\\
\SLO \vecv V_\varepsilon
&:
\Htdiv\rightarrow\Htdiv,
\label{eq:right-V}
\end{align}

In this context, a natual question is if $\pmb\Lambda_{\varepsilon} \SLO$, $\vecv{V}_{\varepsilon} \SLO$, 
$\SLO\pmb\Lambda_{\varepsilon}$, and $\SLO\vecv{V}_{\varepsilon}$ are uniquely solvable, which is what we intend
to explore in the subsequent analysis.


\subsection{Well-posedness of the OSRC-preconditioned EFIEs}
\label{posedness}

We first note that some of the arguments developed in this section
for the EtM formulation are closely related to earlier results on
regularized combined-field integral equations. In particular,
Boubendir and Turc~\cite{boubendir2014well} analyzed EFIE-based and
principal-symbol regularizers and established second-kind properties
for the resulting CFIER formulations. The EtM OSRC approximation
considered here has a closely related pseudodifferential structure.

We nevertheless include the relevant symbol calculations and
compactness arguments explicitly. Besides providing a self-contained
analysis, this places the EtM and MtE formulations within a common
framework and makes their left and right compositions with the EFIE
directly comparable. In particular, the symbol calculations expose
the role of the tangential rotations in the different compositions.
This local description is complemented later by the spectral analysis,
which provides a global characterization of the action of the
operators on the $\Htdiv$ space.

We begin by recalling the local pseudodifferential structure of the
EFIE in the notation used throughout this work.

\begin{proposition}\label{prop:symbol_efie}

Let $\pmb{\xi}=(\xi_2,\xi_3)$ denote the dual variable in local
coordinates on $\Gamma$. Following the parameter-dependent
pseudodifferential analysis of
\cite{darbas2004preconditionneurs,boubendir2014well}, the local Fourier
symbol of the EFIE operator $\SLO$ can be written, in the present
notation, as
\begin{equation}
\label{eq:symbol_SLO}
\sigma_{\SLO}(\pmb\xi)
=
\frac{i}
{2\kappa\sqrt{\|\pmb\xi\|^2-\kappa^2}}
\,
\mathbb J
\left(
\vecv{gd}_{\pmb\xi,\Gamma}
-\kappa^2\mathbb I
\right).
\end{equation}

Equivalently,
\begin{equation}
\label{eq:symbol_SLO_cc}
\sigma_{\SLO}(\pmb\xi)
=
-\frac{i}
{2\kappa\sqrt{\|\pmb\xi\|^2-\kappa^2}}
\left(
\kappa^2\mathbb I
-
\vecv{cc}_{\pmb\xi,\Gamma}
\right)\mathbb J .
\end{equation}

Here
\[
\vecv{gd}_{\pmb\xi,\Gamma}
=
\begin{bmatrix}
\xi_2^2 & \xi_2\xi_3\\
\xi_2\xi_3 & \xi_3^2
\end{bmatrix},
\qquad
\vecv{cc}_{\pmb\xi,\Gamma}
=
\begin{bmatrix}
\xi_3^2 & -\xi_2\xi_3\\
-\xi_2\xi_3 & \xi_2^2
\end{bmatrix},
\]
and
\[
\mathbb J
=
\begin{bmatrix}
0 & -1\\
1 & 0
\end{bmatrix}.
\]

\end{proposition}

\begin{proof}

Using
\[
\sigma_{\Psi_\kappa}(\pmb\xi)
=
\frac{1}
{2\sqrt{\|\pmb\xi\|^2-\kappa^2}},
\qquad
\sigma_{\sgrad}(\pmb\xi)
=
i(\xi_2,\xi_3)^T,
\qquad
\sigma_{\divg}(\pmb\xi)
=
i(\xi_2,\xi_3),
\]
we obtain
\[
\sigma_{\sgrad\circ\Psi_\kappa\circ\divg}(\pmb\xi)
=
-\frac{
\vecv{gd}_{\pmb\xi,\Gamma}
}
{2\sqrt{\|\pmb\xi\|^2-\kappa^2}}.
\]

Since
\[
\SLO
=
\mathcal J
\left(
-i\kappa\pmb\Psi_\kappa
+
(i\kappa)^{-1}
\sgrad\circ\Psi_\kappa\circ\divg
\right),
\]
and $\mathcal J$ is locally represented by $\mathbb J$, it follows that
\[
\sigma_{\SLO}(\pmb\xi)
=
\frac{i}
{2\kappa\sqrt{\|\pmb\xi\|^2-\kappa^2}}
\mathbb J
\left(
\vecv{gd}_{\pmb\xi,\Gamma}
-
\kappa^2\mathbb I
\right).
\]

Finally, using
\[
\mathbb J\vecv{gd}_{\pmb\xi,\Gamma}
=
\vecv{cc}_{\pmb\xi,\Gamma}\mathbb J,
\]
we obtain \eqref{eq:symbol_SLO_cc}.
\end{proof}

\noindent Now we show the following lemma.

\begin{Lem}\label{lemma:EtM_efie}
Let $\Gamma$ be a smooth and closed surface. Then the
formulations~\eqref{eq:left-lambda}--\eqref{eq:right-V} are compact
perturbations of boundedly invertible operators on $\Htdiv$.
Moreover, their principal parts converge to $\pm\frac12\pmb I$ as
$\kappa\to\infty$. Hence, the formulations asymptotically approach
the classical second-kind structure in the high-frequency regime.
\end{Lem}

\begin{proof}

From \cite{el2014approximate}, together with the definitions of
$\pmb\Lambda_\varepsilon$ and $\vecv V_\varepsilon$, their local
symbols are respectively
\begin{align*}
\sigma_{\pmb\Lambda_\varepsilon}(\pmb\xi)
&=
-\frac{i}{\kappa_\varepsilon}
\frac{
\kappa_\varepsilon^2\mathbb I
-
\vecv{cc}_{\pmb\xi,\Gamma}
}{
\sqrt{\|\pmb\xi\|^2-\kappa_\varepsilon^2}
}
\mathbb J,
\\
\sigma_{\vecv V_\varepsilon}(\pmb\xi)
&=
-i\kappa_\varepsilon
\sqrt{\|\pmb\xi\|^2-\kappa_\varepsilon^2}
\,
\mathbb J
\left(
\kappa_\varepsilon^2\mathbb I
-
\vecv{cc}_{\pmb\xi,\Gamma}
\right)^{-1}.
\end{align*}

Moreover, from Proposition~\ref{prop:symbol_efie},
\[
\sigma_{\SLO}(\pmb\xi)
=
\frac{i}
{2\kappa\sqrt{\|\pmb\xi\|^2-\kappa^2}}
\mathbb J
\left(
\vecv{gd}_{\pmb\xi,\Gamma}
-
\kappa^2\mathbb I
\right).
\]

These expressions reveal two distinct regularization mechanisms.
At the level of their local symbols, the EtM approximation has the
same pseudodifferential structure as the EFIE, so that its composition
with the EFIE can be interpreted as a Calderón-type regularization,
similarly to the principal-symbol regularizers considered in
\cite{boubendir2014well}. In contrast, the MtE approximation exhibits
the corresponding inverse structure and can therefore be interpreted
as a local approximate inverse of the EFIE. The following symbol calculations
make these relations explicit. Using
\[
\mathbb J^2=-\mathbb I,
\qquad
\mathbb J\vecv{gd}_{\pmb\xi,\Gamma}
=
\vecv{cc}_{\pmb\xi,\Gamma}\mathbb J,
\qquad
\mathbb J\vecv{cc}_{\pmb\xi,\Gamma}
=
\vecv{gd}_{\pmb\xi,\Gamma}\mathbb J,
\]
the symbols of the compositions are given by
\begin{align}
\sigma_{\pmb\Lambda_\varepsilon\SLO}(\pmb\xi)
&=
\frac{1}{2\kappa\kappa_\varepsilon}
\frac{
\left(
\kappa_\varepsilon^2\mathbb I
-
\vecv{cc}_{\pmb\xi,\Gamma}
\right)
\left(
\kappa^2\mathbb I
-
\vecv{gd}_{\pmb\xi,\Gamma}
\right)
}{
\sqrt{\|\pmb\xi\|^2-\kappa_\varepsilon^2}
\sqrt{\|\pmb\xi\|^2-\kappa^2}
},
\label{eq:Le_symbol}
\\
\sigma_{\vecv V_\varepsilon\SLO}(\pmb\xi)
&=
\frac{\kappa_\varepsilon}{2\kappa}
\frac{
\sqrt{\|\pmb\xi\|^2-\kappa_\varepsilon^2}
}{
\sqrt{\|\pmb\xi\|^2-\kappa^2}
}
\left(
\kappa_\varepsilon^2\mathbb I
-
\vecv{gd}_{\pmb\xi,\Gamma}
\right)^{-1}
\left(
\kappa^2\mathbb I
-
\vecv{gd}_{\pmb\xi,\Gamma}
\right),
\label{eq:VTe_symbol}
\\
\sigma_{\SLO\pmb\Lambda_\varepsilon}(\pmb\xi)
&=
\frac{1}{2\kappa\kappa_\varepsilon}
\frac{
\left(
\kappa^2\mathbb I
-
\vecv{cc}_{\pmb\xi,\Gamma}
\right)
\left(
\kappa_\varepsilon^2\mathbb I
-
\vecv{gd}_{\pmb\xi,\Gamma}
\right)
}{
\sqrt{\|\pmb\xi\|^2-\kappa_\varepsilon^2}
\sqrt{\|\pmb\xi\|^2-\kappa^2}
},
\label{eq:Le_symbol2}
\\
\sigma_{\SLO\vecv V_\varepsilon}(\pmb\xi)
&=
\frac{\kappa_\varepsilon}{2\kappa}
\frac{
\sqrt{\|\pmb\xi\|^2-\kappa_\varepsilon^2}
}{
\sqrt{\|\pmb\xi\|^2-\kappa^2}
}
\left(
\kappa^2\mathbb I
-
\vecv{cc}_{\pmb\xi,\Gamma}
\right)
\left(
\kappa_\varepsilon^2\mathbb I
-
\vecv{cc}_{\pmb\xi,\Gamma}
\right)^{-1}.
\label{eq:VTe_symbol2}
\end{align}

We first analyse the scalar factors. We define
\[
\delta=2i\varepsilon\kappa-\varepsilon^2,
\]
so that
\[
\kappa_\varepsilon^2=\kappa^2+\delta.
\]
As $\|\pmb\xi\|\to\infty$,
\begin{align}
\frac{1}{
\sqrt{
(\|\pmb\xi\|^2-\kappa_\varepsilon^2)
(\|\pmb\xi\|^2-\kappa^2)
}}
&=
\frac{1}{\|\pmb\xi\|^2-\kappa^2}
\left(
1-
\frac{\delta}{\|\pmb\xi\|^2-\kappa^2}
\right)^{-\half}
\nonumber\\
&=
\frac{1}{\|\pmb\xi\|^2-\kappa^2}
\left(
1+
\frac{\delta}{
2(\|\pmb\xi\|^2-\kappa^2)
}
+
\mathcal O(\|\pmb\xi\|^{-4})
\right).
\label{eq:symbol1_rewrite}
\end{align}

At the same time,
\begin{align}
\frac{
\sqrt{\|\pmb\xi\|^2-\kappa_\varepsilon^2}
}{
\sqrt{\|\pmb\xi\|^2-\kappa^2}
}
&=
\sqrt{
1-
\frac{\delta}{
\|\pmb\xi\|^2-\kappa^2
}}
\nonumber\\
&=
1-
\frac{\delta}{
2(\|\pmb\xi\|^2-\kappa^2)
}
+
\mathcal O(\|\pmb\xi\|^{-4}).
\label{eq:symbol_rewrite}
\end{align}

For the vector factors, the Helmholtz projectors satisfy
\[
\vecv{cc}_{\pmb\xi,\Gamma}
=
\|\pmb\xi\|^2
\pmb\Pi_{\mathrm{curl}}(\pmb\xi),
\qquad
\vecv{gd}_{\pmb\xi,\Gamma}
=
\|\pmb\xi\|^2
\pmb\Pi_{\mathrm{grad}}(\pmb\xi),
\]
and
\[
\mathbb I
=
\pmb\Pi_{\mathrm{grad}}(\pmb\xi)
+
\pmb\Pi_{\mathrm{curl}}(\pmb\xi).
\]
Hence,
\[
\alpha\mathbb I
-
\vecv{cc}_{\pmb\xi,\Gamma}
=
\alpha\pmb\Pi_{\mathrm{grad}}
+
(\alpha-\|\pmb\xi\|^2)
\pmb\Pi_{\mathrm{curl}},
\]
and
\[
\alpha\mathbb I
-
\vecv{gd}_{\pmb\xi,\Gamma}
=
(\alpha-\|\pmb\xi\|^2)
\pmb\Pi_{\mathrm{grad}}
+
\alpha\pmb\Pi_{\mathrm{curl}}.
\]

Thus,
\begin{align*}
&
\left(
\kappa_\varepsilon^2\mathbb I
-
\vecv{cc}_{\pmb\xi,\Gamma}
\right)
\left(
\kappa^2\mathbb I
-
\vecv{gd}_{\pmb\xi,\Gamma}
\right)
 =
\kappa_\varepsilon^2
(\kappa^2-\|\pmb\xi\|^2)
\pmb\Pi_{\mathrm{grad}}
+
\kappa^2
(\kappa_\varepsilon^2-\|\pmb\xi\|^2)
\pmb\Pi_{\mathrm{curl}},
\end{align*}
whereas
\begin{align*}
&
\left(
\kappa^2\mathbb I
-
\vecv{cc}_{\pmb\xi,\Gamma}
\right)
\left(
\kappa_\varepsilon^2\mathbb I
-
\vecv{gd}_{\pmb\xi,\Gamma}
\right)
 =
\kappa^2
(\kappa_\varepsilon^2-\|\pmb\xi\|^2)
\pmb\Pi_{\mathrm{grad}}
+
\kappa_\varepsilon^2
(\kappa^2-\|\pmb\xi\|^2)
\pmb\Pi_{\mathrm{curl}}.
\end{align*}

Furthermore,
\begin{align*}
&
\left(
\kappa_\varepsilon^2\mathbb I
-
\vecv{gd}_{\pmb\xi,\Gamma}
\right)^{-1}
\left(
\kappa^2\mathbb I
-
\vecv{gd}_{\pmb\xi,\Gamma}
\right)
 =
\frac{\kappa^2-\|\pmb\xi\|^2}
{\kappa_\varepsilon^2-\|\pmb\xi\|^2}
\pmb\Pi_{\mathrm{grad}}
+
\frac{\kappa^2}{\kappa_\varepsilon^2}
\pmb\Pi_{\mathrm{curl}},
\end{align*}
and
\begin{align*}
&
\left(
\kappa^2\mathbb I
-
\vecv{cc}_{\pmb\xi,\Gamma}
\right)
\left(
\kappa_\varepsilon^2\mathbb I
-
\vecv{cc}_{\pmb\xi,\Gamma}
\right)^{-1}
 =
\frac{\kappa^2}{\kappa_\varepsilon^2}
\pmb\Pi_{\mathrm{grad}}
+
\frac{\kappa^2-\|\pmb\xi\|^2}
{\kappa_\varepsilon^2-\|\pmb\xi\|^2}
\pmb\Pi_{\mathrm{curl}}.
\end{align*}

As $\|\pmb\xi\|\to\infty$,
\[
\frac{
\kappa_\varepsilon^2-\|\pmb\xi\|^2
}{
\kappa^2-\|\pmb\xi\|^2
}
=
1+\mathcal O(\|\pmb\xi\|^{-2}),
\qquad
\frac{
\kappa^2-\|\pmb\xi\|^2
}{
\kappa_\varepsilon^2-\|\pmb\xi\|^2
}
=
1+\mathcal O(\|\pmb\xi\|^{-2}).
\]

Therefore, \eqref{eq:Le_symbol} and \eqref{eq:VTe_symbol}
can be written respectively as
\begin{align*}
\sigma_{\pmb\Lambda_\varepsilon\SLO}(\pmb\xi)
&=
-\left(
\frac{\kappa_\varepsilon}{2\kappa}
\pmb\Pi_{\mathrm{grad}}
+
\frac{\kappa}{2\kappa_\varepsilon}
\pmb\Pi_{\mathrm{curl}}
\right)
+
\mathcal O(\|\pmb\xi\|^{-2}),
\\
\sigma_{\vecv V_\varepsilon\SLO}(\pmb\xi)
&=
\left(
\frac{\kappa_\varepsilon}{2\kappa}
\pmb\Pi_{\mathrm{grad}}
+
\frac{\kappa}{2\kappa_\varepsilon}
\pmb\Pi_{\mathrm{curl}}
\right)
+
\mathcal O(\|\pmb\xi\|^{-2}).
\end{align*}

Similarly, \eqref{eq:Le_symbol2} and \eqref{eq:VTe_symbol2}
give
\begin{align*}
\sigma_{\SLO\pmb\Lambda_\varepsilon}(\pmb\xi)
&=
-\left(
\frac{\kappa}{2\kappa_\varepsilon}
\pmb\Pi_{\mathrm{grad}}
+
\frac{\kappa_\varepsilon}{2\kappa}
\pmb\Pi_{\mathrm{curl}}
\right)
+
\mathcal O(\|\pmb\xi\|^{-2}),
\\
\sigma_{\SLO\vecv V_\varepsilon}(\pmb\xi)
&=
\left(
\frac{\kappa}{2\kappa_\varepsilon}
\pmb\Pi_{\mathrm{grad}}
+
\frac{\kappa_\varepsilon}{2\kappa}
\pmb\Pi_{\mathrm{curl}}
\right)
+
\mathcal O(\|\pmb\xi\|^{-2}).
\end{align*}

The remainder symbols are of order $-2$ on the nonharmonic
components, so the corresponding remainder operators $\vecv K_j$
are pseudodifferential operators of order $-2$. Their compactness 
follows the standard regularity and compact
embedding mechanism used for electromagnetic boundary integral
operators; see, for example, \cite{boubendir2014well}. 
Here we make the argument explicit.

By the standard Sobolev mapping property of pseudodifferential operators on smooth
compact manifolds~\cite{taylor1991pseudodifferential}, an operator of order $m$
extends continuously from $\mathrm H^s(\Gamma)$ to $\mathrm H^{s-m}(\Gamma)$.
The same mapping property holds for operators acting on tangential
vector fields. Hence,
\[
\vecv K_j:
\mathbf H^{-1/2}_\times(\Gamma)
\longrightarrow
\mathbf H^{3/2}_\times(\Gamma)
\]
continuously. Moreover, since $\divg$ is a differential operator
of order one, the composition $\divg\vecv K_j$ is of order $-1$,
and therefore
\[
\divg\vecv K_j:
\mathbf H^{-1/2}_\times(\Gamma)
\longrightarrow
\mathrm H^{1/2}(\Gamma)
\]
continuously.

Consequently,
\[
\vecv K_j:\Htdiv\longrightarrow \mathbf X,
\qquad
\mathbf X:=
\left\{
\vecv v\in\mathbf H^{3/2}_\times(\Gamma):
\divg\vecv v\in \mathrm H^{1/2}(\Gamma)
\right\},
\]
continuously, where $\mathbf X$ is equipped with its natural graph norm.

Since $\Gamma$ is smooth and compact, the Rellich compactness
theorem for Sobolev spaces on compact manifolds
\cite{mclean2000strongly} implies

\[
\mathbf H^{3/2}_\times(\Gamma)
\subseteq
\mathbf H^{-1/2}_\times(\Gamma),
\qquad
\mathrm H^{1/2}(\Gamma)
\subseteq
\mathrm H^{-1/2}(\Gamma).
\]
These embeddings imply that $\mathbf X$ is compactly embedded in
$\Htdiv$. Indeed, every bounded sequence $(\vecv v_{\mathrm n})$ in $\mathbf X$
admits a subsequence such that both $\vecv v_{\mathrm n}$ and
$\divg\vecv v_{\mathrm n}$ converge in the corresponding
$\mathrm H^{-1/2}$ spaces. By the definition of the graph norm
\[
\|\vecv v\|_{\Htdiv}^2
=
\|\vecv v\|_{\mathbf H^{-1/2}_\times(\Gamma)}^2
+
\|\divg\vecv v\|_{H^{-1/2}(\Gamma)}^2,
\]
this subsequence converges in $\Htdiv$. Hence
\[
\mathbf X\subseteq\Htdiv,
\]
and therefore each
$\vecv K_j:\Htdiv\to\Htdiv$ is compact.

For a surface of nonzero genus, the Hodge decomposition of
$\Htdiv$ also contains the finite-dimensional harmonic
subspace $\mathcal Ha(\Gamma)$, which is not detected by the local
principal-symbol analysis. Any operator supported on
$\mathcal Ha(\Gamma)$ is finite-rank and therefore compact. Hence any
harmonic contribution may be absorbed into the compact remainder.

Let
\[
\vecv I
=
\pmb\Pi_{\mathrm{grad}}
+
\pmb\Pi_{\mathrm{curl}}
+
\pmb\Pi_{\mathcal Ha}
\]
denote the Hodge decomposition of $\Htdiv$. We define
\[
\vecv P_\varepsilon^{\,\mathrm L}
:=
\frac{\kappa_\varepsilon}{2\kappa}
\pmb\Pi_{\mathrm{grad}}
+
\frac{\kappa}{2\kappa_\varepsilon}
\pmb\Pi_{\mathrm{curl}}
+
\frac12\pmb\Pi_{\mathcal Ha},
\]
and
\[
\vecv P_\varepsilon^{\,\mathrm R}
:=
\frac{\kappa}{2\kappa_\varepsilon}
\pmb\Pi_{\mathrm{grad}}
+
\frac{\kappa_\varepsilon}{2\kappa}
\pmb\Pi_{\mathrm{curl}}
+
\frac12\pmb\Pi_{\mathcal Ha}.
\]

For fixed $\kappa>0$, these operators are diagonal with respect to
the Hodge decomposition. Their eigenvalues on the gradient, curl,
and harmonic components are
\[
\frac{\kappa_\varepsilon}{2\kappa},
\qquad
\frac{\kappa}{2\kappa_\varepsilon},
\qquad
\frac12,
\]
with the first two appearing in opposite order for
$\vecv P_\varepsilon^{\,\mathrm L}$ and
$\vecv P_\varepsilon^{\,\mathrm R}$. Since
$\kappa_\varepsilon\neq0$, both operators are boundedly invertible on
$\Htdiv$.

Hence the four preconditioned operators can be written as
\[
\pmb\Lambda_\varepsilon\SLO
=
-\vecv P_\varepsilon^{\,\mathrm L}+\vecv K_1,
\qquad
\vecv V_\varepsilon\SLO
=
\vecv P_\varepsilon^{\,\mathrm L}+\vecv K_2,
\]
and
\[
\SLO\pmb\Lambda_\varepsilon
=
-\vecv P_\varepsilon^{\,\mathrm R}+\vecv K_3,
\qquad
\SLO\vecv V_\varepsilon
=
\vecv P_\varepsilon^{\,\mathrm R}+\vecv K_4,
\]
where $\vecv K_j$ are compact operators on $\Htdiv$.

Therefore, each of the four preconditioned operators is a compact
perturbation of a boundedly invertible operator on $\Htdiv$.

Moreover, with
\[
\varepsilon=0.39\,\kappa^{1/3}\mathcal H_0^{2/3},
\]
and $\mathcal H_0>0$ fixed, one has
\[
\frac{\varepsilon}{\kappa}
=
\mathcal O(\kappa^{-2/3}),
\]
and therefore
\[
\frac{\kappa_\varepsilon}{2\kappa}
=
\frac12+\mathcal O(\kappa^{-2/3}),
\qquad
\frac{\kappa}{2\kappa_\varepsilon}
=
\frac12+\mathcal O(\kappa^{-2/3}).
\]
Consequently,
\[
\vecv P_\varepsilon^{\,\mathrm L}
=
\frac12\vecv I+\mathcal O(\kappa^{-2/3}),
\qquad
\vecv P_\varepsilon^{\,\mathrm R}
=
\frac12\vecv I+\mathcal O(\kappa^{-2/3}),
\]
as $\kappa\to\infty$.

Hence,
\[
\pmb\Lambda_\varepsilon\SLO,
\quad
\SLO\pmb\Lambda_\varepsilon
\]
have principal parts tending to $-\frac12\vecv I$, whereas
\[
\vecv V_\varepsilon\SLO,
\quad
\SLO\vecv V_\varepsilon
\]
have principal parts tending to $+\frac12\vecv I$.

\end{proof}

\begin{Lem}
\label{Leminjectivity}

Let $\Gamma$ be a sufficiently smooth, closed surface, and assume that
$\kappa$ is not an interior Maxwell eigenvalue. Then the regularized EFIE
operators~\eqref{eq:left-lambda}--\eqref{eq:right-V} are injective.

\end{Lem}

\begin{proof}

We first show that $\pmb\Lambda_\varepsilon$ and
$\vecv V_\varepsilon$ are injective.

Let $\{(\lambda_m,\phi_m)\}_{m\geq 1}$ be the eigenpairs of the
scalar Laplace--Beltrami operator \eqref{eq:lap_belt} on $\Gamma$ , i.e.,
\[
\Deltag\phi_m=\lambda_m\phi_m.
\]
Using the Hodge decomposition, a tangential field
$\vecv u\in\Htdiv$ can be written as
\begin{equation}
\vecv u
=
\sum_{m\geq 1}
a_m\vcurlg\phi_m
+
b_m\sgrad\phi_m
+
\vecv h,
\label{eq:hodge_decomp}
\end{equation}

where $\vecv h$ is a harmonic tangential vector field.

From the definition of the vector Hodge Laplacian in~\eqref{eq:hodge_lap},
together with
\[
\scurl\sgrad\phi_m=0,
\qquad
\divg\vcurlg\phi_m=0,
\qquad
\Delta_\Gamma\phi_m
=
-\scurl\vcurlg\phi_m,
\]
we obtain
\[
\vDeltag\sgrad\phi_m
=
\sgrad\Delta_\Gamma\phi_m
=
\lambda_m\sgrad\phi_m,
\]
and
\[
\vDeltag\vcurlg\phi_m
=
\vcurlg\Delta_\Gamma\phi_m
=
\lambda_m\vcurlg\phi_m.
\]
Moreover, for a harmonic tangential field $\vecv h$,
\[
\vDeltag\vecv h=0.
\]

Therefore,
\begin{equation}
\label{eq:l1-eig}
\begin{cases}
\pmb\Lambda_{1,\varepsilon}\sgrad\phi_m
&=
\left(
1+\dfrac{\lambda_m}{\kappa_\varepsilon^2}
\right)^{\half}
\sgrad\phi_m,
\\[2mm]
\pmb\Lambda_{1,\varepsilon}\vcurlg\phi_m
&=
\left(
1+\dfrac{\lambda_m}{\kappa_\varepsilon^2}
\right)^{\half}
\vcurlg\phi_m,
\end{cases}
\end{equation}
whereas
\begin{equation}
\label{eq:l2-eig}
\begin{cases}
\pmb\Lambda_{2,\varepsilon}\sgrad\phi_m
&=
\sgrad\phi_m,
\\[2mm]
\pmb\Lambda_{2,\varepsilon}\vcurlg\phi_m
&=
\left(
1+\dfrac{\lambda_m}{\kappa_\varepsilon^2}
\right)
\vcurlg\phi_m.
\end{cases}
\end{equation}

It follows that
\[
\pmb\Lambda_{1,\varepsilon}^{-1}
\pmb\Lambda_{2,\varepsilon}
\sgrad\phi_m
=
\left(
1+\frac{\lambda_m}{\kappa_\varepsilon^2}
\right)^{-\half}
\sgrad\phi_m,
\]
and
\[
\pmb\Lambda_{1,\varepsilon}^{-1}
\pmb\Lambda_{2,\varepsilon}
\vcurlg\phi_m
=
\left(
1+\frac{\lambda_m}{\kappa_\varepsilon^2}
\right)^{\half}
\vcurlg\phi_m.
\]

Using
\[
\mathcal J\sgrad\phi_m
=
\vcurlg\phi_m,
\qquad
\mathcal J\vcurlg\phi_m
=
-\sgrad\phi_m,
\]
and the definition
\[
\pmb\Lambda_\varepsilon
=
\pmb\Lambda_{1,\varepsilon}^{-1}
\pmb\Lambda_{2,\varepsilon}
\mathcal J,
\]
we obtain
\begin{equation}
\label{eq:l-eig}
\begin{cases}
\pmb\Lambda_\varepsilon\sgrad\phi_m
&=
\left(
1+\dfrac{\lambda_m}{\kappa_\varepsilon^2}
\right)^{\half}
\vcurlg\phi_m,
\\[2mm]
\pmb\Lambda_\varepsilon\vcurlg\phi_m
&=
-
\left(
1+\dfrac{\lambda_m}{\kappa_\varepsilon^2}
\right)^{-\half}
\sgrad\phi_m.
\end{cases}
\end{equation}

Similarly,
\[
\pmb\Lambda_{2,\varepsilon}^{-1}
\pmb\Lambda_{1,\varepsilon}
\sgrad\phi_m
=
\left(
1+\frac{\lambda_m}{\kappa_\varepsilon^2}
\right)^{\half}
\sgrad\phi_m,
\]
and
\[
\pmb\Lambda_{2,\varepsilon}^{-1}
\pmb\Lambda_{1,\varepsilon}
\vcurlg\phi_m
=
\left(
1+\frac{\lambda_m}{\kappa_\varepsilon^2}
\right)^{-\half}
\vcurlg\phi_m.
\]
Hence, from
\[
\vecv V_\varepsilon
=
-\mathcal J
\pmb\Lambda_{2,\varepsilon}^{-1}
\pmb\Lambda_{1,\varepsilon},
\]
we obtain
\begin{equation}
\label{eq:v-eig}
\begin{cases}
\vecv V_\varepsilon\sgrad\phi_m
&=
-
\left(
1+\dfrac{\lambda_m}{\kappa_\varepsilon^2}
\right)^{\half}
\vcurlg\phi_m,
\\[2mm]
\vecv V_\varepsilon\vcurlg\phi_m
&=
\left(
1+\dfrac{\lambda_m}{\kappa_\varepsilon^2}
\right)^{-\half}
\sgrad\phi_m.
\end{cases}
\end{equation}

Thus, each subspace
\(
\operatorname{span}
\left\{
\sgrad\phi_m,\vcurlg\phi_m
\right\}
\)
is invariant under both operators, $\pmb\Lambda_\varepsilon$ and
$\vecv V_\varepsilon$. In the ordered basis
$\{\sgrad\phi_m,\vcurlg\phi_m\}$, their matrices are respectively
\[
\begin{bmatrix}
0
&
-\left(
1+\dfrac{\lambda_m}{\kappa_\varepsilon^2}
\right)^{-\half}
\\[3mm]
\left(
1+\dfrac{\lambda_m}{\kappa_\varepsilon^2}
\right)^{\half}
&
0
\end{bmatrix}, \qquad
\begin{bmatrix}
0
&
\left(
1+\dfrac{\lambda_m}{\kappa_\varepsilon^2}
\right)^{-\half}
\\[3mm]
-
\left(
1+\dfrac{\lambda_m}{\kappa_\varepsilon^2}
\right)^{\half}
&
0
\end{bmatrix}.
\]
Since $\varepsilon>0$, one has
\(
1+\frac{\lambda_m}{\kappa_\varepsilon^2}\neq0,
\)
and both matrices are invertible. Hence neither
$\pmb\Lambda_\varepsilon$ nor $\vecv V_\varepsilon$ has a
non-trivial kernel on the gradient and curl components. It remains to consider the harmonic subspace. Since
\[
\vDeltag\vecv h=0,
\]
and
\[
\divg\vecv h=0,
\qquad
\scurl\vecv h=0,
\]
we have
\[
\pmb\Lambda_{1,\varepsilon}\vecv h
=
\vecv h,
\qquad
\pmb\Lambda_{2,\varepsilon}\vecv h
=
\vecv h.
\]
Therefore,
\[
\pmb\Lambda_\varepsilon\vecv h
=
\mathcal J\vecv h,
\qquad
\vecv V_\varepsilon\vecv h
=
-\mathcal J\vecv h.
\]

Since the rotation $\mathcal J$ is invertible on the harmonic
subspace, it follows that
\[
\ker(\pmb\Lambda_\varepsilon)
=
\ker(\vecv V_\varepsilon)
=
\{0\}.
\]
Thus both $\pmb\Lambda_\varepsilon$ and $\vecv V_\varepsilon$
are injective.

Moreover, under the non-resonance assumption, the EFIE operator
$\SLO$ is injective, i.e.,
\[
\SLO\vu=0
\quad\Longrightarrow\quad
\vu=0.
\]

Therefore, the compositions~\eqref{eq:left-lambda}-\eqref{eq:right-V} are injective.

\end{proof}

\begin{Cor}\label{cor:wellposedness}
Under the assumptions of Lemmas~\ref{lemma:EtM_efie}
and~\ref{Leminjectivity}, the regularized EFIE
operators~\eqref{eq:left-lambda}--\eqref{eq:right-V}
are boundedly invertible on $\Htdiv$. Consequently, for every
$\vecv f\in\Htdiv$, each corresponding regularized EFIE admits a
unique solution $\vecv u\in\Htdiv$.
\end{Cor}

\begin{proof}
By Lemma~\ref{lemma:EtM_efie}, each of the four operators can be
written as
\[
\vecv A=\vecv P+\vecv K,
\]
where $\vecv P$ is boundedly invertible and $\vecv K$ is compact.
Since $\vecv P$ is boundedly invertible,
\[
\vecv A
=
\vecv P
\left(
\vecv I+\vecv P^{-1}\vecv K
\right),
\]
and $\vecv P^{-1}\vecv K$ is compact.

By Lemma~\ref{Leminjectivity}, $\vecv A$ is injective. Hence, if
\[
\left(
\vecv I+\vecv P^{-1}\vecv K
\right)\vecv u=0,
\]
then
\[
\vecv A\vecv u
=
\vecv P
\left(
\vecv I+\vecv P^{-1}\vecv K
\right)\vecv u
=0,
\]
and therefore $\vecv u=0$. Thus
$\vecv I+\vecv P^{-1}\vecv K$ is injective.

Since $\vecv P^{-1}\vecv K$ is compact, the Fredholm alternative
implies that
$\vecv I+\vecv P^{-1}\vecv K$ is also surjective, and hence
bijective. Because $\vecv P$ is bijective, it follows that
$\vecv A$ is bijective.

Finally, since $\vecv A$ is a bounded bijective operator on $\Htdiv$,
the bounded inverse theorem implies that $\vecv A^{-1}$ is bounded.
Hence $\vecv A$ is boundedly invertible.

Consequently, for every $\vecv f\in\Htdiv$, there exists a unique
$\vecv u\in\Htdiv$ such that
\[
\vecv A\vecv u=\vecv f.
\]
\end{proof}

Having established the well-posedness of the four OSRC-preconditioned
formulations, we now turn to the numerical approximation of the
fractional surface operators entering the EtM and MtE maps.

\subsection{Weak formulations for Pad\'e approximants of the MtE and EtM maps}

The remaining question is how these operators can be effectively computed 
in a numerical setting, since the main implementational difficulty lies in the operators
$\pmb\Lambda_{1,\varepsilon}^{\pm}$, that involve the fractional
surface operators
\[
\left(
\vecv I+\frac{\vDeltag}{\kappa_\varepsilon^2}
\right)^{\pm \half}.
\]

To enable their numerical evaluation, El Bouajaji et
al.~\cite{el2014approximate} and Chaillat et
al.~\cite{CDLLJCP2017} proposed stable approximations based on Pad\'e
rational approximants. For completeness we have included these expressions in the appendices~\ref{sect:app1}
and~~\ref{sect:app2}. These make it possible to replace the operator $\pmb\Lambda_{1,\varepsilon}
=
\left(
\vecv I+\frac{\vDeltag}{\kappa_\varepsilon^2}
\right)^{\half}$ by

\begin{equation}
\label{eq:pade_approximant}
\tilde{\pmb{\Lambda}}_{1,\varepsilon} = \left(\vecv{I}\mathrm R_{0}-\sum_{\ell=1}^{\mathrm{L}_{1}}\frac{\mathrm{A}_{\ell}}{\mathrm{B}_{\ell}} \left(\vecv{I}+\mathrm{B}_{\ell}\frac{\vDeltag}{\kappa_{\varepsilon}^2}\right)^{-1}\right),
\end{equation}

and $\pmb{\Lambda}_{1,\varepsilon}^{-1}
=
\left(
\vecv I+\frac{\vDeltag}{\kappa_\varepsilon^2}
\right)^{-\half}$ by

\begin{equation}
\label{eq:pade_approximant2}
\tilde{\pmb{\Lambda}}_{1,\varepsilon}^{-1} =  \sum_{\ell=0}^{\mathrm{L}_{2}-1} \mathrm{C}_{\ell}\left(\mathrm{D}_{\ell}\vecv{I}+\frac{\vDeltag}{\kappa_{\varepsilon}^2}\right)^{-1},
\end{equation}

The Pad\'e approximations~\eqref{eq:pade_approximant}
and~\eqref{eq:pade_approximant2} show that the numerical evaluation of
both $\tilde{\pmb\Lambda}_{1,\varepsilon}$ and
$\tilde{\pmb\Lambda}_{1,\varepsilon}^{-1}$ reduces to the solution of
a finite number of shifted surface Hodge laplacian problems. Rather
than treating the two operators separately, we introduce a generic
shifted resolvent of the Hodge laplacian operator of the form
\begin{equation}
\vecv{R}_{\ell}
:=
\left(
\alpha_\ell\vecv I
+
\beta_\ell\frac{\vDeltag}{\kappa_\varepsilon^2}
\right)^{-1},
\label{eq:shift_lap}
\end{equation}
where the coefficients $\alpha_\ell$ and $\beta_\ell$ depend on the
Pad\'e approximation under consideration. Thus, given a tangential field $\vecv b$, let
\[
\pmb\phi^\ell
=
\vecv R_\ell\vecv b.
\]
Then $\pmb\phi^\ell$ satisfies
\[
\left(
\alpha_\ell\vecv I
+
\beta_\ell\frac{\vDeltag}{\kappa_\varepsilon^2}
\right)
\pmb\phi^\ell
=
\vecv b.
\]
Using \eqref{eq:hodge_lap} this equation can be written as
\begin{equation}
\label{eq:generic_shifted_laplacian}
\left[
\alpha_\ell\vecv I
+
\beta_\ell
\vnablag
\frac{1}{\kappa_\varepsilon^2}
\divg
-
\beta_\ell
\vscurl
\frac{1}{\kappa_\varepsilon^2}
\scurl
\right]
\pmb\phi^\ell
=
\vecv b.
\end{equation}

To solve this equation, we obtain a mixed variational formulation of
\eqref{eq:generic_shifted_laplacian} by introducing the auxiliary scalar
variable
\[
\rho^\ell
=
\frac{\beta_\ell}{\kappa_\varepsilon^2}
\divg\pmb\phi^\ell.
\]

For sufficiently smooth fields, given
$\vecv b\in\Htcurl$, find
\(
\pmb\phi^\ell\in\Htcurl,
\;
\rho^\ell\in H^{\half}(\Gamma),
\) such that, for all
$\vecv w\in\Htcurl$ and
$z\in \mathrm H^{\half}(\Gamma)$,
\begin{equation}
\label{eq:generic_resolvent_wf}
\begin{cases}
\alpha_\ell
\displaystyle
\int_\Gamma
\pmb\phi^\ell\cdot\vecv w\,d\Gamma
+
\displaystyle
\int_\Gamma
\vnablag\rho^\ell\cdot\vecv w\,d\Gamma
-
\beta_\ell
\displaystyle
\int_\Gamma
\frac{1}{\kappa_\varepsilon^2}
\scurl\pmb\phi^\ell
\,
\scurl\vecv w\,d\Gamma
=
\displaystyle
\int_\Gamma
\vecv b\cdot\vecv w\,d\Gamma,
\\[3ex]
\displaystyle
\int_\Gamma
\kappa_\varepsilon^2
\rho^\ell z\,d\Gamma
+
\beta_\ell
\displaystyle
\int_\Gamma
\pmb\phi^\ell\cdot\vnablag z\,d\Gamma
=
0.
\end{cases}
\end{equation}

For $\tilde{\pmb\Lambda}_{1,\varepsilon}$, taking
$\alpha_\ell=1$ and $\beta_\ell=\mathrm{B}_\ell$ gives
\[
\pmb\phi^\ell
=
\left(
\vecv I
+
\mathrm{B}_\ell\frac{\vDeltag}{\kappa_\varepsilon^2}
\right)^{-1}
\vecv b,
\]
and hence
\[
\tilde{\pmb\Lambda}_{1,\varepsilon}\vecv b
=
\mathrm{R}_0\vecv b
-
\sum_{\ell=1}^{\mathrm{L}_{1}}
\frac{\mathrm{A}_\ell}{\mathrm{B}_\ell}\pmb\phi^\ell.
\]

Similarly, for
$\tilde{\pmb\Lambda}_{1,\varepsilon}^{-1}$, taking
$\alpha_\ell=\mathrm{D}_\ell$ and $\beta_\ell=1$ gives
\[
\pmb\phi^\ell
=
\left(
\mathrm{D}_\ell\vecv I
+
\frac{\vDeltag}{\kappa_\varepsilon^2}
\right)^{-1}
\vecv b,
\]
so that
\[
\tilde{\pmb\Lambda}_{1,\varepsilon}^{-1}\vecv b
=
\sum_{\ell=0}^{\mathrm{L}_{2}-1}
\mathrm{C}_\ell\pmb\phi^\ell.
\]

The remaining factor, $\pmb\Lambda_{2,\varepsilon}$, does not involve a
fractional surface operator and can therefore be treated directly. Its
weak form is given by the bilinear form
\begin{equation}
\label{eq:lambda2_wf}
\tilde{\pmb\Lambda}_{2,\varepsilon} (\vecv r,\vecv v) :=
\int_\Gamma
\vecv r\cdot\vecv v\,d\Gamma
-
\int_\Gamma
\frac{1}{\kappa_\varepsilon^2}
\scurl\vecv r\,
\scurl\vecv v
\,d\Gamma,
\qquad
\vecv r,\vecv v\in\Htcurl.
\end{equation}

We have thus reduced the evaluation of the generalized OSRC maps to
standard mixed surface problems involving only integer-order
differential operators. The remaining issue is how these weak
operators interact with the Galerkin representation of the EFIE,
which we address in the following section.

\section{Boundary Element Method implementations}\label{sect:bem_impl}

Given the weak formulations derived above for
$\tilde{\pmb{\Lambda}}_{1,\varepsilon}$,
$\tilde{\pmb{\Lambda}}_{1,\varepsilon}^{-1}$, and
$\pmb{\Lambda}_{2,\varepsilon}$, we now discuss some aspects of their
boundary element discretization. These considerations are particularly
important for determining whether the resulting discrete operators can
be applied as refinement-free preconditioners.

As usual, we consider a triangular discretization $\Gamma_h$ of
$\Gamma$. We introduce the standard pair of discrete spaces employed in
the discretization of the EFIE: the Rao--Wilton--Glisson (RWG) space,
which is $\Htdiv$-conforming, and its rotated counterpart, the Scaled
N\'ed\'elec (SNC) space, which is $\Htcurl$-conforming.

We also consider the Buffa--Christiansen (BC) space and its rotated
counterpart, the Rotated Buffa--Christiansen (RBC) space. These spaces
are constructed as linear combinations of RWG and SNC basis functions respectively
on the barycentric refinement of $\Gamma_h$. Thus, BC and RBC provide
stable dual pairings with the respective SNC and RWG spaces. Their
introduction is necessary because the $\mathrm L^2(\Gamma)$ inner product between the
RWG and SNC spaces results in singular matrices by construction. The BC and RBC spaces therefore provide suitable
$\Htdiv$- and $\Htcurl$-conforming dual spaces for the corresponding
discretizations~\cite{buffa2007dual}.

Finally, we employ the space of continuous piecewise-linear functions
(P1) for the discretization of scalar functions in
$\mathrm{H}^{\half}(\Gamma)$.

Before presenting the discrete forms of
$\pmb{\Lambda}_{\varepsilon}$ and $\vecv V_{\varepsilon}$, we briefly
recall some details on the algebraic representation and composition of
Galerkin operators that will be useful in the discussion that follows.

\subsection{A quick note regarding operator composition in BEM}

Having introduced the discrete spaces used throughout the numerical
implementation, we briefly recall their algebraic representations. We start 
by distinguishing throughout between coefficient and load vectors.

Let $\mathcal V_h^1$ and $\mathcal V_h^2$ denote finite-dimensional
trial and test spaces, with bases
\[
\{\vecv q_j\}_{j=1}^{N_1},
\qquad
\{\vecv r_i\}_{i=1}^{N_2},
\]
respectively. Depending on the operator under consideration, these may
be chosen from the RWG, SNC, BC, RBC or P1 spaces introduced above. 
Moreover, if
\[
\vecv u_h
=
\sum_{j=1}^{N_1}c_j\vecv q_j,
\]
then $\mathbf c=(c_j)$ is the coefficient vector of $\vecv u_h$.
For a bilinear form $a(\cdot,\cdot)$, the corresponding Galerkin
matrix is
\[
\mathsf A_{ij}
=
a(\vecv q_j,\vecv r_i),
\]
and maps trial coefficients to test-space loads,
\[
\mathbf f=\mathsf A\mathbf c,
\qquad
\vecv f_i=a(\vecv u_h,\vecv r_i).
\]
Thus, although $\mathbf c$ and $\mathbf f$ may have the same
dimension, they represent different algebraic objects. The Galerkin
matrix represents the discrete mapping
\[
\mathsf A:
\mathbb C^{N_1}
\longrightarrow
\mathbb C^{N_2},
\qquad
\mathbf c\longmapsto\mathbf f,
\]
from coefficient vectors in $\mathcal V_h^1$ to load vectors
associated with $\mathcal V_h^2$.
If $\mathsf A$ is square and nonsingular, this mapping can be inverted,
\[
\mathbf c=\mathsf A^{-1}\mathbf f.
\]
The resulting vector $\mathbf c=(c_j)_{j=1}^{N_1}$ determines the
Galerkin approximation
\[
\vecv u_h
=
\sum_{j=1}^{N_1}c_j\vecv q_j
\in\mathcal V_h^1.
\]

Hence, the interpretation of an algebraic vector as a discrete field
is intrinsically tied to the basis in which its coefficients are
represented.

\medskip
A particularly important example is the $\mathrm L^2(\Gamma)$ inner product,
\[
a(\vecv u_h,\vecv v_h)
=
\int_\Gamma
\vecv u_h\cdot\vecv v_h\,d\Gamma.
\]
The corresponding Gram matrix
$\mathsf G\in\mathbb C^{N_2\times N_1}$ has entries
\[
\mathsf G_{ij}
=
\int_\Gamma
\vecv q_j\cdot\vecv r_i\,d\Gamma.
\]
If $N_1=N_2$ and $\mathsf G$ is nonsingular, its inverse provides the
corresponding load-to-coefficient map,
\[
\mathbf c
=
\mathsf G^{-1}\mathbf f.
\]
Thus, $\mathsf G^{-1}$ identifies a load vector associated with the
test space $\mathcal V_h^2$ with the coefficient vector of the
corresponding discrete field in the trial space $\mathcal V_h^1$.

These distinctions are particularly relevant when composing operators. Consider two operators $\vecv A$ and $\vecv B$,
with Galerkin matrices $\mathsf A$ and $\mathsf B$. The output of
$\mathsf A$ is naturally a load vector, whereas $\mathsf B$ acts on a
coefficient vector in its trial space. Hence, algebraic compatibility
of the matrix dimensions alone is not sufficient to represent the
composition $\vecv B\vecv A$.

If the test space associated with $\mathsf A$ is stably paired with
the trial space associated with $\mathsf B$, the corresponding Gram
matrix $\mathsf G$ provides the required load-to-coefficient
identification. The Galerkin representation of the composition is
therefore
\[
\mathsf B\,\mathsf G^{-1}\mathsf A,
\]
rather than simply
\[
\mathsf B\,\mathsf A.
\]

This is precisely the issue addressed by the Calder\'on
Multiplicative Preconditioner 
\cite{andriulli2008multiplicative}, where BC basis functions are
introduced to obtain a stable invertible duality pairing. The same
consideration will be important below when composing the EFIE with the
OSRC preconditioners and concluding if their implementation calls for a mesh refinement or not.

\subsection{Galerkin discretization of the OSRC operators}

We now apply these considerations to the discretization of both the MtE and the EtM. Since both
$\tilde{\pmb\Lambda}_{1,\varepsilon}$ and
$\tilde{\pmb\Lambda}_{1,\varepsilon}^{-1}$ are constructed from the resolvent of the 
shifted Laplacian, it is sufficient to first
consider the finite element implementation of \eqref{eq:shift_lap}.
This is pretty much straigthforward from equation~\eqref{eq:generic_resolvent_wf}. Let $\{\vecv t_j\}$ and $\{\vecv r_i\}$ denote bases of the chosen
$\Htcurl$-conforming trial and test spaces, respectively, and let
$\{\eta_j\}$ and $\{\lambda_i\}$ denote bases of the corresponding P1
spaces. We introduce the matrices
\begin{equation}
\begin{aligned}
\mathsf G_{ij}
&=
\int_{\Gamma_h}
\vecv t_j\cdot\vecv r_i\,d\Gamma_h,
\qquad
(\mathsf N_\varepsilon)_{ij}
=
\int_{\Gamma_h}
\frac{1}{\kappa_\varepsilon^2}
\hscurl\vecv t_j\,
\hscurl\vecv r_i\,d\Gamma_h,
\\[1ex]
(\mathsf K_\varepsilon)_{ij}
&=
\int_{\Gamma_h}
\kappa_\varepsilon^2
\eta_j\lambda_i\,d\Gamma_h,
\qquad
\mathsf L_{ij}
=
\int_{\Gamma_h}
\hsgrad\eta_j\cdot\vecv r_i\,d\Gamma_h.
\end{aligned}
\label{eq:discrete_resolvent_matrices}
\end{equation}

For a given load vector $\mathbf b$, the application of 
$\vecv R_\ell$ is equivalent to solving
\begin{equation}
\label{eq:generic_matrix_system}
\begin{bmatrix}
\alpha_\ell\mathsf G-\beta_\ell\mathsf N_\varepsilon
&
\mathsf L
\\
\beta_\ell\mathsf L^T
&
\mathsf K_\varepsilon
\end{bmatrix}
\begin{bmatrix}
\boldsymbol\phi^\ell
\\
\boldsymbol\rho^\ell
\end{bmatrix}
=
\begin{bmatrix}
\mathbf b
\\
\mathbf 0
\end{bmatrix}.
\end{equation}
Here $\boldsymbol\phi^\ell$ and $\boldsymbol\rho^\ell$ are the
coefficient vectors of the discrete fields $\pmb\phi_h^\ell$ and
$\rho_h^\ell$, respectively. Hence, $\vecv R_\ell$ 
maps the load vector $\mathbf b$ to the coefficient vector
$\boldsymbol\phi^\ell$. Thus, given a load vector $\mathbf b$,
the action of the $\tilde{\mathsf\Lambda}_{1,\varepsilon}^{-1}$ is given by
\[
\tilde{\mathsf\Lambda}_{1,\varepsilon}^{-1}\mathbf b
=
\sum_{\ell=0}^{\mathrm{L}_{2}-1}
\mathrm{C}_\ell\boldsymbol\phi^\ell,
\]
which is itself a coefficient vector.

For $\tilde{\mathsf\Lambda}_{1,\varepsilon}$, the applications of $\vecv R_\ell$ 
likewise produce coefficient vectors $\boldsymbol\phi^\ell$. Since
these contributions must be combined with the load vector
$\mathrm{R}_0\mathbf b$, they are mapped back to load-vector representations
through the corresponding Gram matrix. Hence,
\[
\tilde{\mathsf\Lambda}_{1,\varepsilon}\mathbf b
=
\mathrm{R}_0\mathbf b
-
\sum_{\ell=1}^{\mathrm{L}_{1}}
\frac{\mathrm{A}_\ell}{\mathrm{B}_\ell}
\mathsf G\boldsymbol\phi^\ell.
\]
Consequently,
$\tilde{\mathsf\Lambda}_{1,\varepsilon}$ maps load vectors to load
vectors, whereas
$\tilde{\mathsf\Lambda}_{1,\varepsilon}^{-1}$ maps load vectors to
coefficient vectors.

The remaining factor $\pmb\Lambda_{2,\varepsilon}$ is discretized
directly from~\eqref{eq:lambda2_wf}, giving
\[
\mathsf\Lambda_{2,\varepsilon}
=
\mathsf G-\mathsf N_\varepsilon.
\]
Accordingly, $\mathsf\Lambda_{2,\varepsilon}$ maps a coefficient vector in its trial space to a load vector
associated with its test space. The mapping properties of these operators are summarized in table~\ref{tab:etm_maps}.

\begin{table}[htbp]
\centering
\begin{tabular}{|c|c|c|}
\hline
Operator & Input & Output \\
\hline
$\tilde{\mathsf\Lambda}_{1,\varepsilon}$
& Load vector & Load vector \\
\hline
$\tilde{\mathsf\Lambda}_{1,\varepsilon}^{-1}$
& Load vector & Coefficient vector \\
\hline
$\mathsf\Lambda_{2,\varepsilon}$
& Coefficient vector & Load vector \\
\hline
$\mathsf\Lambda_{2,\varepsilon}^{-1}$
& Load vector & Coefficient vector \\
\hline
$\mathsf G$
& Coefficient vector & Load vector \\
\hline
$\mathsf G^{-1}$
& Load vector & Coefficient vector \\
\hline
\end{tabular}
\caption{Algebraic mapping properties of the principal discrete
operators entering the generalized OSRC preconditioners.}
\label{tab:etm_maps}
\end{table}

\subsection{Refinement-free preconditioners for the EFIE}\label{sect:ref_free}

Here we show that the continuous compositions~\eqref{eq:left-lambda}-\eqref{eq:right-V} do not translate directly
into products of the matrices arising from their standard Galerkin
discretizations. In particular, the tangential rotations appearing in
the continuous definitions may be absorbed by, or introduced through,
the choice of trial and test spaces. We therefore distinguish between
the continuous operators that we wish to compose and the weak operators
that are actually assembled.

Let

\begin{minipage}{0.55\textwidth}
\begin{equation}
\vecv A_{\kappa}
=
-i\kappa\pmb\Psi_\kappa
+
(i\kappa)^{-1}
\sgrad\circ\Psi_\kappa\circ\divg
\label{eq:A-kappa}
\end{equation}
\end{minipage}
\hfill
\begin{minipage}{0.4\textwidth}
\[
\vecv B_{\kappa_\varepsilon}
=
\pmb\Lambda_{1,\varepsilon}^{-1}
\pmb\Lambda_{2,\varepsilon}.
\]
\end{minipage}

\bigskip

Then
\[
\SLO=\mathcal J\vecv A_{\kappa},
\qquad
\pmb\Lambda_\varepsilon
=
\vecv B_{\kappa_\varepsilon}\mathcal J,
\qquad
\vecv V_\varepsilon
=
-\mathcal J\vecv B_{\kappa_\varepsilon}^{-1}.
\]

Let $\{\vecv f_j\}$ and $\{\vecv t_j\}$ denote the RWG and SNC bases,
respectively, with
\begin{equation}
\vecv t_j=\mathcal J\vecv f_j.
\label{eq:rwg-to-snc}
\end{equation}
Then
\[
(\SLO\vecv f_j,\vecv t_i)_\Gamma
=
(\mathcal J\vecv A_\kappa\vecv f_j,
 \mathcal J\vecv f_i)_\Gamma
=
(\vecv A_\kappa\vecv f_j,\vecv f_i)_\Gamma.
\]
Consequently, the standard RWG/SNC EFIE matrix is algebraically the
weak Galerkin matrix associated with $\vecv A_\kappa$, namely
\[
(\mathsf S_\kappa)_{ij}
=
(\vecv A_\kappa\vecv f_j,\vecv f_i)_\Gamma
=
-i\kappa
(\pmb\Psi_\kappa\vecv f_j,\vecv f_i)_\Gamma
-
\frac{1}{i\kappa}
(\Psi_\kappa\divg\vecv f_j,\divg\vecv f_i)_\Gamma.
\]
Thus, although the continuous EFIE contains the tangential rotation
$\mathcal J$, its standard assembled weak form represents
$\vecv A_\kappa$.

The OSRC operators are instead discretized using SNC basis functions.
Using~\eqref{eq:rwg-to-snc}, their assembled representations satisfy\footnote{Here and below, the symbol $\sim$ indicates the continuous rotation
structure represented by the corresponding Galerkin matrix, rather
than equality of operators.}
\begin{equation}
\mathsf\Lambda_\varepsilon
\sim
\mathcal J\pmb\Lambda_\varepsilon
=
\mathcal J
\vecv B_{\kappa_\varepsilon}
\mathcal J,
\label{eq:discrete-lambda}
\end{equation}
and
\begin{equation}
\mathsf V_\varepsilon
\sim
\vecv V_\varepsilon\mathcal J
=
-\mathcal J
\vecv B_{\kappa_\varepsilon}^{-1}
\mathcal J.
\label{eq:discrete-V}
\end{equation}

We can now compare the discrete products with the continuous
compositions.
For left MtE preconditioning,
\[
\mathsf V_\varepsilon\mathsf S_\kappa
\sim
-\mathcal J
\vecv B_{\kappa_\varepsilon}^{-1}
\mathcal J
\vecv A_\kappa.
\]
On the other hand,
\[
\vecv V_\varepsilon\SLO
=
(-\mathcal J\vecv B_{\kappa_\varepsilon}^{-1})
(\mathcal J\vecv A_\kappa)
=
-\mathcal J
\vecv B_{\kappa_\varepsilon}^{-1}
\mathcal J
\vecv A_\kappa.
\]
Hence
\[
\boxed{
\mathsf V_\varepsilon\mathsf S_\kappa
\sim
\vecv V_\varepsilon\SLO.
}
\]
The left-preconditioned discrete system therefore reproduces directly
the rotation structure of the continuous MtE composition.

For right MtE preconditioning, one must additionally account for the
algebraic mapping properties of $\mathsf V_\varepsilon$.
As discussed above, $\mathsf V_\varepsilon$ acts on an SNC load vector
and returns a coefficient vector. Since the unknown of the iterative
solver is represented by coefficients, we introduce the ordinary
SNC Gram matrix
\[
\mathsf G_{ij}
=
(\vecv t_j,\vecv t_i)_\Gamma,
\]
which maps SNC coefficient vectors to SNC load vectors. The
right-preconditioned discrete operator is therefore
\[
\mathsf S_\kappa
\mathsf V_\varepsilon
\mathsf G.
\]
Its rotation structure is
\[
\mathsf S_\kappa
\mathsf V_\varepsilon
\mathsf G
\sim
-\vecv A_\kappa
\mathcal J
\vecv B_{\kappa_\varepsilon}^{-1}
\mathcal J.
\]
Meanwhile,
\[
\SLO\vecv V_\varepsilon
=
-\mathcal J
\vecv A_\kappa
\mathcal J
\vecv B_{\kappa_\varepsilon}^{-1},
\]
and therefore
\[
-\mathcal J
(\SLO\vecv V_\varepsilon)
\mathcal J
=
-\vecv A_\kappa
\mathcal J
\vecv B_{\kappa_\varepsilon}^{-1}
\mathcal J.
\]
Thus,
\[
\boxed{
\mathsf S_\kappa
\mathsf V_\varepsilon
\mathsf G
\sim
-\mathcal J
(\SLO\vecv V_\varepsilon)
\mathcal J.
}
\]
The remaining rotations are external changes between the RWG and SNC
representations and are accounted for by
the use of SNC basis functions (see equation \eqref{eq:rwg-to-snc}). Moreover, $\mathsf G$ is an ordinary SNC--SNC
Gram matrix and is assembled entirely on the primal mesh. Hence neither
left nor right MtE preconditioning requires barycentric refinement.

The EtM approximation behaves differently. Direct left composition
gives
\[
\mathsf\Lambda_\varepsilon\mathsf S_\kappa
\sim
\mathcal J
\vecv B_{\kappa_\varepsilon}
\mathcal J
\vecv A_\kappa,
\]
whereas the continuous composition satisfies
\[
\pmb\Lambda_\varepsilon\SLO
=
(\vecv B_{\kappa_\varepsilon}\mathcal J)
(\mathcal J\vecv A_\kappa)
=
-\vecv B_{\kappa_\varepsilon}\vecv A_\kappa.
\]
The two expressions therefore do not differ only by external changes
of tangential representation. An additional rotation is required at
the internal interface between the EFIE and the OSRC operator.

Similarly, direct right composition gives
\[
\mathsf S_\kappa\mathsf\Lambda_\varepsilon
\sim
\vecv A_\kappa
\mathcal J
\vecv B_{\kappa_\varepsilon}
\mathcal J,
\]
whereas
\[
\SLO\pmb\Lambda_\varepsilon
=
\mathcal J
\vecv A_\kappa
\vecv B_{\kappa_\varepsilon}
\mathcal J.
\]

Thus, at discrete level we must find a way to incorporate an extra rotation $\mathcal{J}$ at 
the interface of $\mathsf S_\kappa$ and $\mathsf\Lambda_\varepsilon$. At the same time,
this internal interface must also connect the load vector produced by the first Galerkin 
operator to the coefficient vector required by the second.
We therefore introduce the mixed Gram matrix

\[
(\mathsf G_{\rm mix})_{ij}
=
(\vecv f_j,\mathcal J \vecv q_i)_\Gamma,
\]
and the corrected EtM products then have the form

\begin{minipage}{0.48\textwidth}
\begin{equation}
\mathsf G
\mathsf\Lambda_\varepsilon
\mathsf G_{\rm mix}^{-1}
\mathsf S_\kappa
\label{eq:left-etm}
\end{equation}
\end{minipage}
\hfill
\begin{minipage}{0.48\textwidth}
\begin{equation}
\mathsf S_\kappa
\mathsf G_{\rm mix}^{-1}
\mathsf G
\mathsf\Lambda_\varepsilon
\label{eq:right-etm}
\end{equation}
\end{minipage}
\medskip

Since the RWG--SNC
pairing on the primal mesh is singular, a stable
construction of $\mathsf G_{\rm mix}^{-1}$ requires the BC/RBC spaces
on the barycentric refinement. 

In summary, both left and right MtE preconditioning can be build wihout 
the need of a barycentric refinement of $\Gamma_h$, whereas both EtM 
formulations require the dual BC/RBC construction.

\begin{table}[htbp]
\centering
\begin{tabular}{|c|c|c|}
\hline
& \textbf{LP} & \textbf{RP} \\
\hline
$\mathsf V_{\varepsilon}$
& Refinement-free
& Refinement-free \\
\hline
$\mathsf\Lambda_{\varepsilon}$
& Requires refinement
& Requires refinement \\
\hline
\end{tabular}
\caption{Refinement requirements for the left- and right-preconditioned
EFIE formulations using the MtE and EtM approximations.}
\label{tab:refinement_requirements}
\end{table}

\section{Numerical Examples}

In this section, we investigate numerically the EtM and MtE
approximations and their application to the preconditioning of the
EFIE. 

We start by examining the spectral properties of the EtM and MtE
approximations, both as standalone operators and when composed with
the EFIE from the left and from the right. This allows us to assess
the effect of the discrete formulations introduced in
Section~\ref{sect:ref_free}.


We complement this with the assesment of the performance of the different EFIE
preconditioners in terms of GMRES convergence and computational cost.
The left- and right-preconditioned EtM and MtE formulations are
compared with the unpreconditioned EFIE and the Calderón
multiplicative preconditioner. In addition to iteration counts,
we report assembly and solution times to account for the computational
cost associated with the different discretizations.

When needed, for clarity, we denote by $[\mathsf A]_{\mathrm{X,Y}}$ the Galerkin matrix associated
with an operator $\mathsf A$, discretized using the trial space $\mathrm X$ and the
test space $\mathrm Y$, and introduce the abbreviations
\[
\mathrm R=\mathrm{RWG},\qquad
\mathrm S=\mathrm{SNC},\qquad
\mathrm B=\mathrm{BC},\qquad
\mathrm R^\star=\mathrm{RBC},\qquad
\mathrm S_b=\mathrm{SNC\text{-}b}.
\]
Here, $\mathrm{SNC\text{-}b}$ denotes the SNC basis functions assembled on
the barycentrically refinement of $\Gamma_h$. The RBC
basis functions are constructed from this space through the transformation
$\mathsf T$; see~\cite{buffa2007dual,andriulli2008multiplicative}. We perform all of our implementations in \texttt{bempp-cl}
\cite{scroggs2017software}.

We begin by comparing the spectra of the EtM and MtE approximations
separately on the unit sphere. Figure~\ref{fig:ev-comp}
shows that the two operators exhibit similar spectral distributions, in
agreement with the spectral analysis of Lemma~\ref{Leminjectivity}.
Although this structure is not immediately apparent from the plot, the two
branches of the spectra arise separately from the $\sgrad$ and $\vcurl$
components of the Hodge decomposition~\eqref{eq:hodge_decomp}.

\begin{figure}[H]
\centering
\includegraphics[width=0.5\linewidth]{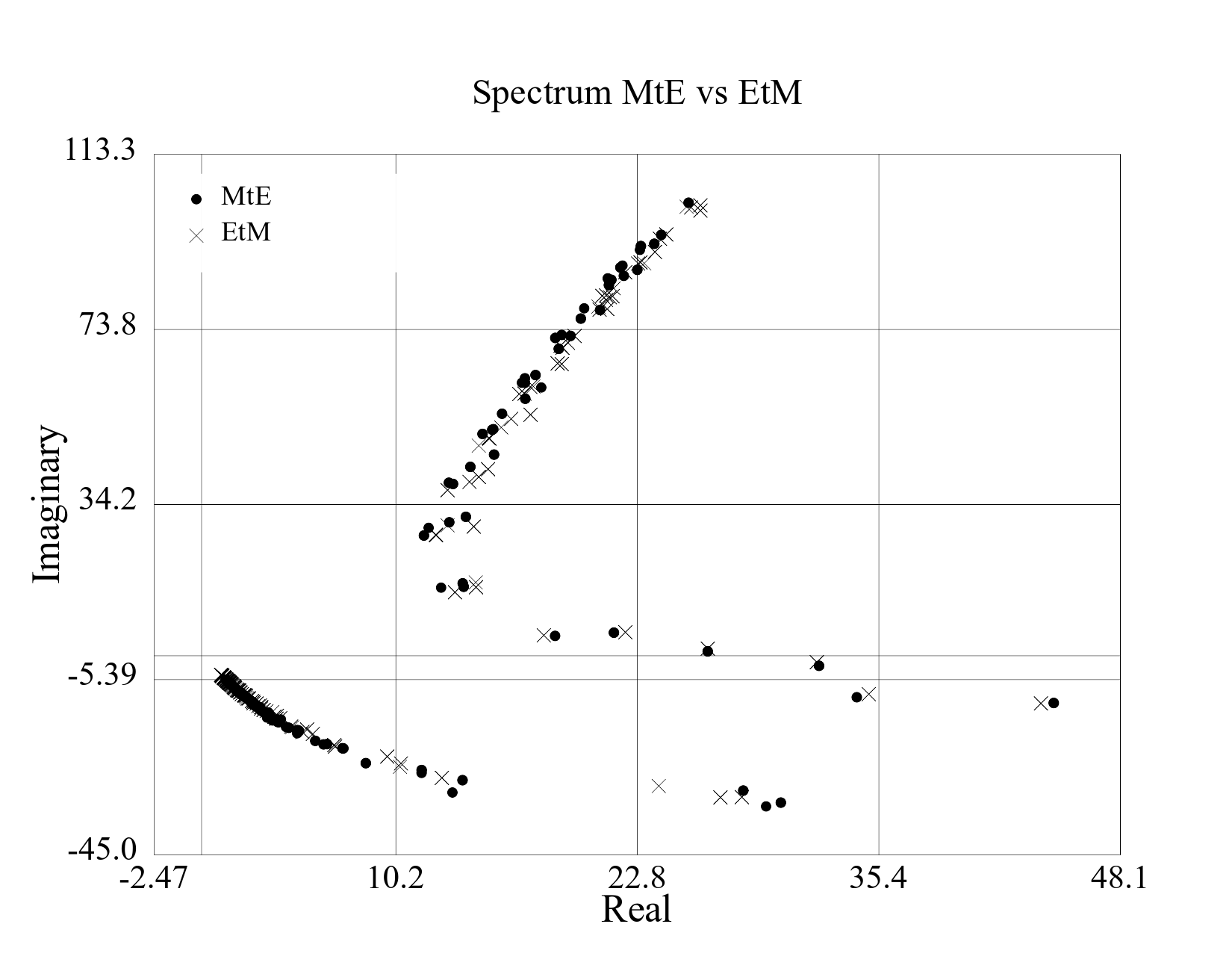}
\caption{MtE vs. EtM spectra.}
\label{fig:ev-comp}
\end{figure}

Figure~\ref{fig:mte-ev-comp} shows the spectra of the left- and
right-preconditioned MtE formulations,
$\mathsf V_{\varepsilon}\mathsf S_\kappa$ and
$\mathsf S_\kappa\mathsf V_{\varepsilon}$. In particular, for
$\kappa=\pi$, two distinct clusters of eigenvalues can be observed,
corresponding to the $\sgrad$ and $\vcurl$ components of the Hodge
decomposition. The spectral clustering suggests that the two
eigenvalue branches of $\mathsf V_{\varepsilon}$ compensate for the
corresponding branches of $\mathsf S_\kappa$. We can understand this behavior
for the particular case of the unit sphere $S_R$ and the spectral decomposition of 
the EFIE on it. Indeed, it is possible to show that (see~\cite{kress1985minimizing})

\begin{figure}[H]
\centering
\includegraphics[width=\linewidth]{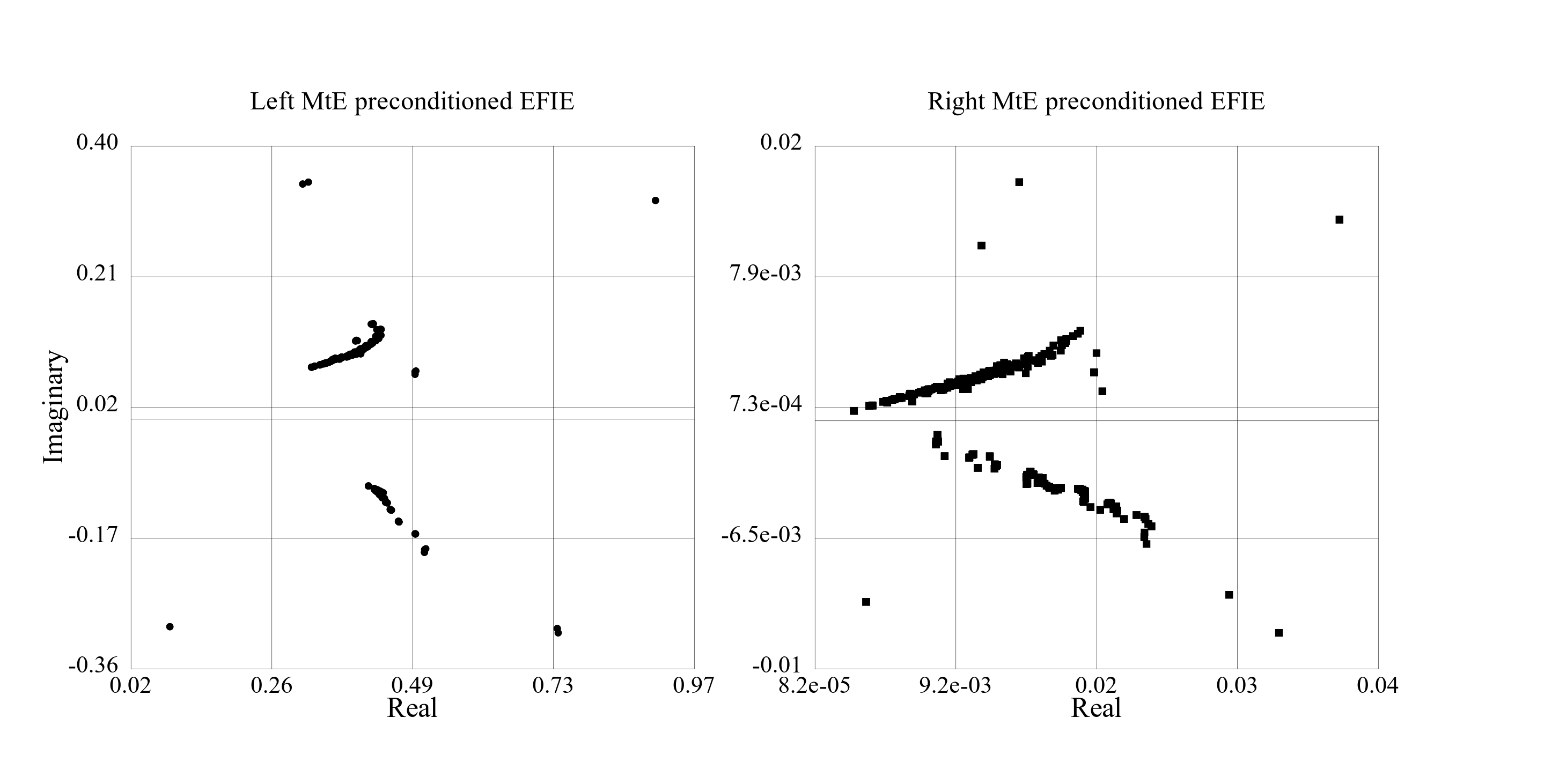}
\caption{Spectra of $\mathsf V_{\varepsilon}\mathsf S_\kappa$ and
$\mathsf S_\kappa\mathsf V_{\varepsilon}$.}
\label{fig:mte-ev-comp}
\end{figure}

\begin{equation}
\begin{cases}
\mathsf S_{\kappa}\left(\cnabla{S_{R}}Y_{m}^{n}\right)
=
t^{-}_{m}\cnabla{S_{R}}Y_{m}^{n}
=
\xi_{m}^{(1)}(\kappa R)\psi_{m}(\kappa R)
\cnabla{S_{R}}Y_{m}^{n},
\\[1ex]
\mathsf S_{\kappa}\left(\vcurl_{S_{R}}Y_{m}^{n}\right)
=
t^{+}_{m}\vcurl_{S_{R}}Y_{m}^{n}
=
\xi_{m}'^{(1)}(\kappa R)\psi_{m}'(\kappa R)
\vcurl_{S_{R}}Y_{m}^{n}.
\end{cases}
\end{equation}

Here, $\xi_{m}^{(1)}(x)=x h_{m}^{(1)}(x)$ and
$\psi_{m}(x)=x j_{m}(x)$ are the Riccati--Hankel and
Riccati--Bessel functions of order $m$, respectively.
For the sphere, the scalar Laplace--Beltrami eigenfunctions are
the spherical harmonics $\mathrm{Y}_m^n$, with
\[
\Delta_{S_R}\mathrm{Y}_m^n=\lambda_m \mathrm{Y}_m^n,
\qquad
\lambda_m=-\frac{m(m+1)}{R^2},
\]
according to the present sign convention.
Therefore, the general spectral relations derived in
Lemma~\ref{Leminjectivity} reduce to explicit formulas on the two
Hodge branches generated by $\cnabla{S_R}\mathrm{Y}_m^n$ and
$\vcurl_{S_R}\mathrm{Y}_m^n$. In particular,
\[
\mathsf V_\varepsilon
\left(\cnabla{S_R}\mathrm{Y}_m^n\right)
=
v_m^+ \vcurl_{S_R}\mathrm{Y}_m^n,
\qquad
v_m^+
=
-
\left(
1-\frac{m(m+1)}{\kappa_\varepsilon^2R^2}
\right)^{1/2},
\]
and
\[
\mathsf V_\varepsilon
\left(\vcurl_{S_R}\mathrm{Y}_m^n\right)
=
v_m^- \cnabla{S_R}\mathrm{Y}_m^n,
\qquad
v_m^-
=
\left(
1-\frac{m(m+1)}{\kappa_\varepsilon^2R^2}
\right)^{-1/2}.
\]

We compare these eigenvalues with the EFIE eigenvalues $t_m^{-}$ and
$t_m^{+}$ in Figure~\ref{fig:ev-branches}. The relations
\[
v_m^+ \sim \frac{1}{t_m^-},
\qquad
v_m^- \sim \frac{1}{t_m^+},
\]
show that each spectral branch of $\mathsf V_\varepsilon$ compensates
for the corresponding branch of $\mathsf S_\kappa$. This corroborates
the preceding spectral analysis.

\begin{figure}[H]
\centering
\includegraphics[width=0.6\linewidth]{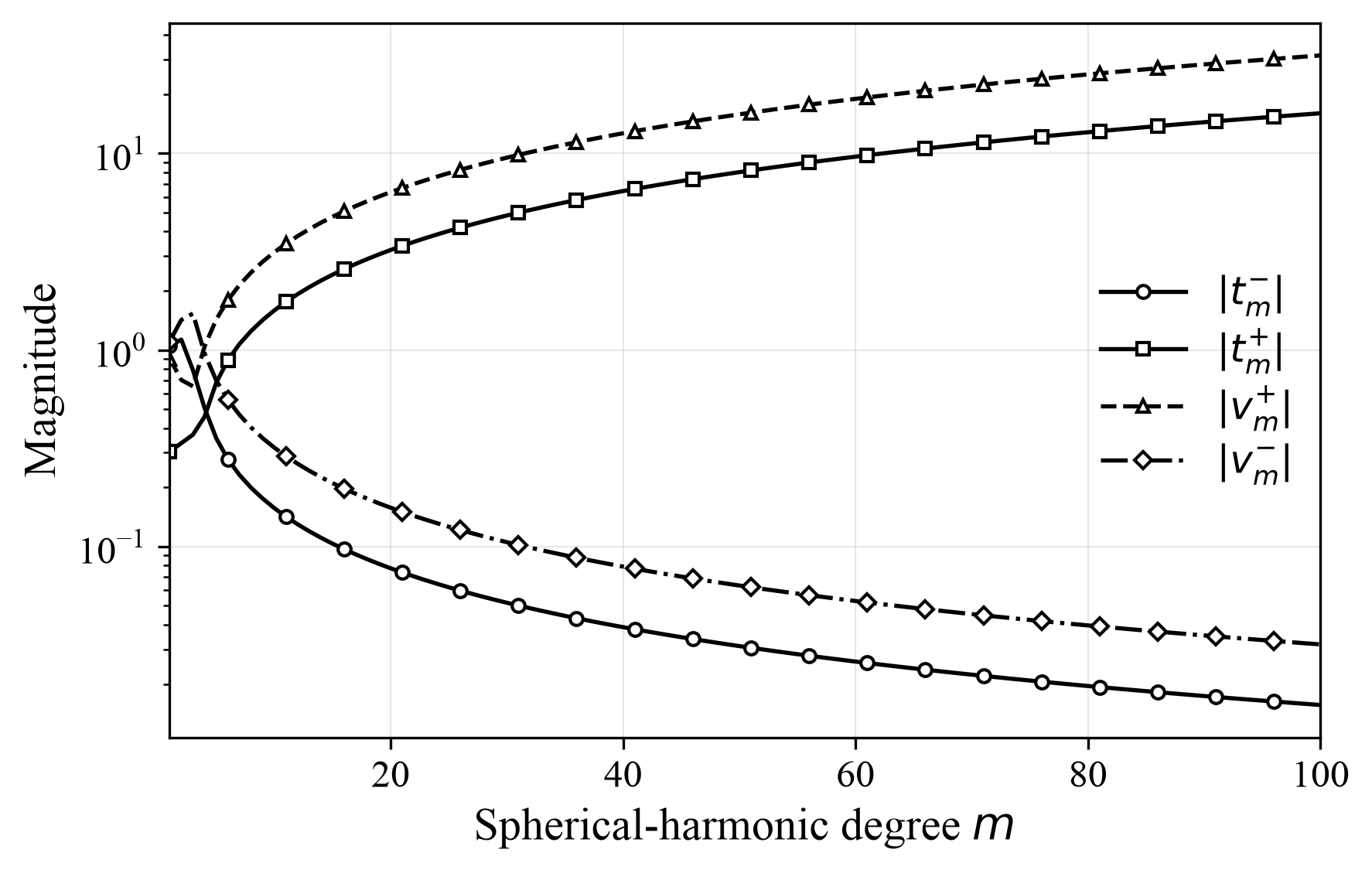}
\caption{Magnitudes of the analytical eigenvalues on a sphere.}
\label{fig:ev-branches}
\end{figure}

This contrasts with the eigenvalues produced by
$\mathsf\Lambda_\varepsilon\mathsf S_\kappa$ and
$\mathsf S_\kappa\mathsf\Lambda_\varepsilon$; that is, when the EtM
approximation is used as a refinement-free preconditioner. Indeed, the
analysis in Lemma~\ref{Leminjectivity} establishes a reciprocal relation
between the eigenvalue branches of $\mathsf V_{\varepsilon}$ and
$\mathsf\Lambda_{\varepsilon}$. Consequently, when no map
$\mathsf G_{\rm mix}^{-1}$ is introduced, the resulting operator has
eigenvalue branches proportional to
\[
t_m^{-}v_m^{-}
\qquad\text{and}\qquad
t_m^{+}v_m^{+},
\]
rather than the compensating pairings
$t_m^{-}v_m^{+}$ and $t_m^{+}v_m^{-}$ obtained for the MtE
formulation, and eigenvalue clustering is lost.

\begin{figure}[H]
\centering
\includegraphics[width=\linewidth]{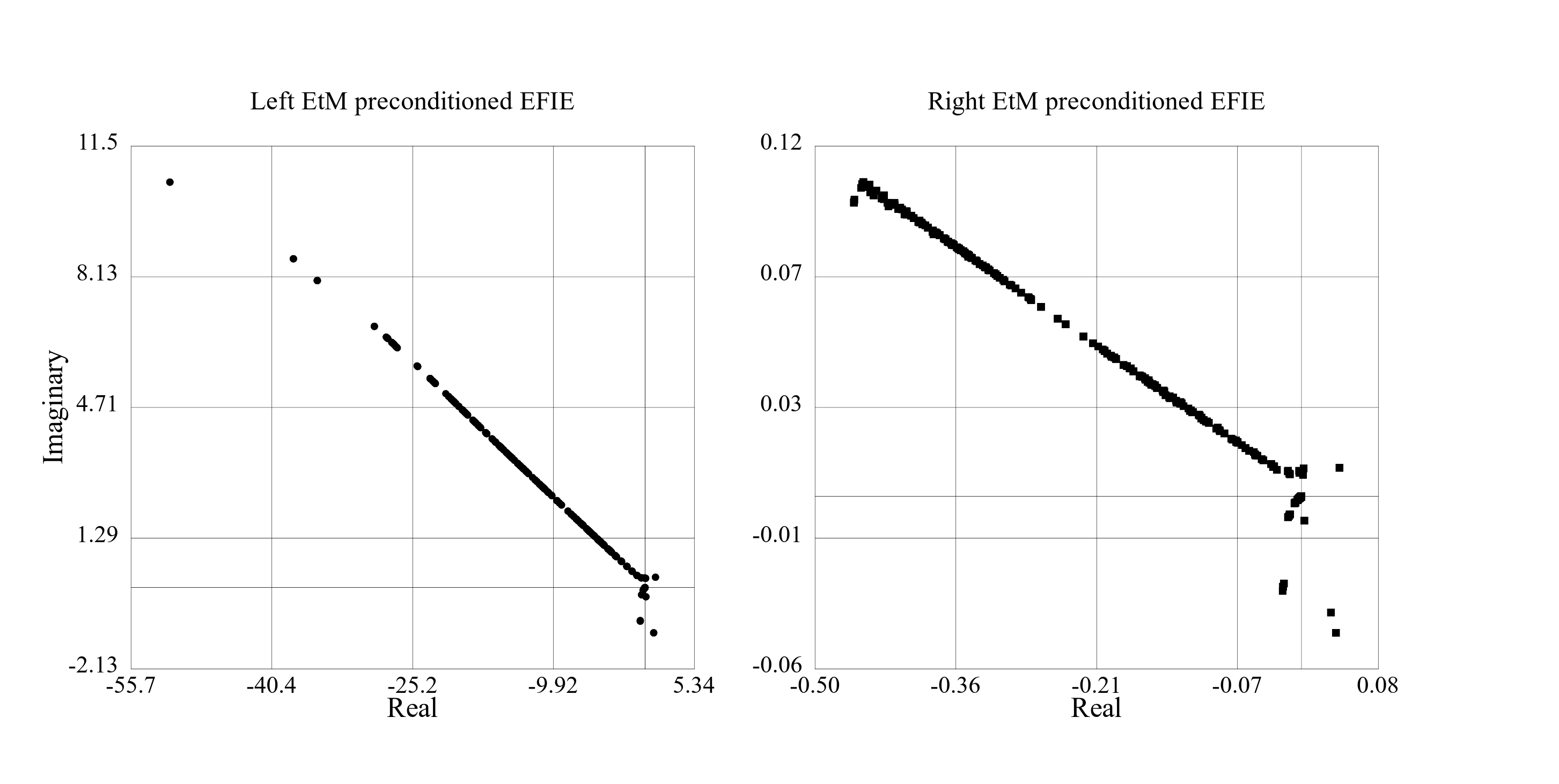}
\caption{Spectra of $\mathsf\Lambda_\varepsilon\mathsf S_\kappa$ and
$\mathsf S_\kappa\mathsf\Lambda_\varepsilon$.}
\label{fig:etm-ev-comp}
\end{figure}

However, the EtM approximation remains effective
in its primary role as an Electric-to-Magnetic map. This can be seen
directly by considering the radar cross section (RCS) obtained from the
corresponding OSRC approximation. Indeed, we can obtain the RWG coefficients 
of $\magnetic  \vecv{e}^{\mathrm s}$ via the transformation

\[
\mathsf G_{\mathrm{R,R^{*}}}^{-1}
\mathsf G_{\mathrm{R,S}}
[\tilde{\mathsf\Lambda}_{1,\varepsilon}^{-1}
 \mathsf\Lambda_{2,\varepsilon}]_{\mathrm {S,S}}(\pmb\nu \times \tangential  \vecv{e}^{\mathrm i})_{\mathrm{S}}.\]

And then we use \eqref{stratton-chu} to compute $\vecv{e}^{\mathrm s}$ 
and obtain an RCS for the scattered field.
Figure~\ref{fig:rcs} compares the RCS obtained using the OSRC EtM
approximation with the analytical solution for $\kappa=\pi$ and
$\kappa=8\pi$. Both experiments use eight points per wavelength and a
Pad\'e approximation of order two. The OSRC EtM captures the overall
behaviour of the analytical RCS over the full range of observation
angles, including the increasingly oscillatory response at the higher
wavenumber.

\begin{figure}[H]
\centering
\includegraphics[width=\linewidth]{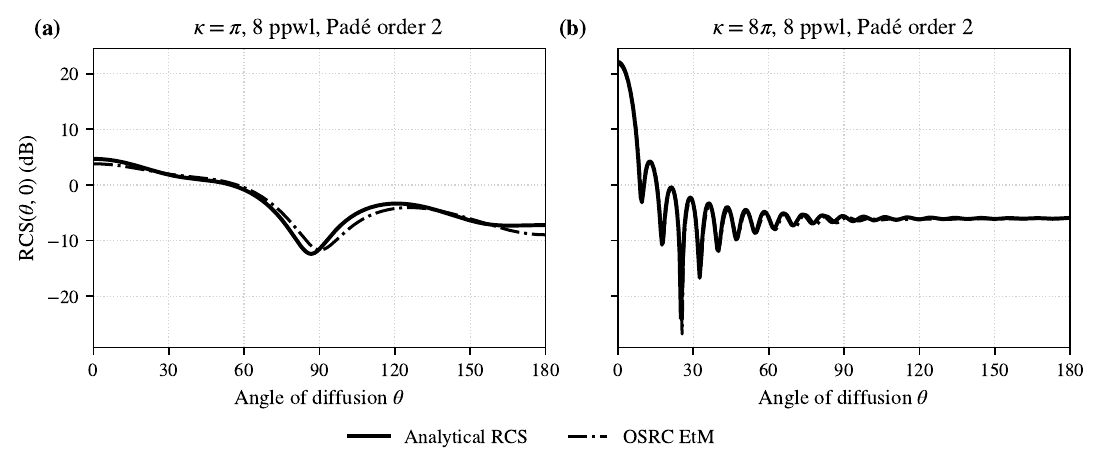}
\caption{Comparison of the analytical and OSRC EtM radar cross sections
for $\kappa=\pi$ and $\kappa=8\pi$, using eight points per wavelength
and a Pad\'e approximation of order two.}
\label{fig:rcs}
\end{figure}

Finally, we assess the performance of the MtE and EtM approximations as
preconditioners for the EFIE in terms of GMRES iteration counts, assembly
times, and solution times. To implement the rotated EtM formulations in
\eqref{eq:left-etm} and \eqref{eq:right-etm}, we employ BC and RBC basis
functions on the barycentrically refined mesh.

We then define the discrete rotated
EtM block as
\begin{equation}
\mathsf L_\varepsilon
:=
\mathsf T^{\mathrm T}
[\mathsf G]_{\mathrm {S_b,S_b}}
[\tilde{\mathsf\Lambda}_{1,\varepsilon}^{-1}
 \mathsf\Lambda_{2,\varepsilon}]_{\mathrm {S_b,S_b}}
\mathsf T.
\label{eq:discrete-etm}
\end{equation}

With this notation, the left- and right-preconditioned EtM systems are
assembled as
\begin{align}
\mathsf P_{\mathrm{EtM}}^{\mathrm L}
&=
\mathsf L_\varepsilon
[\mathsf G]_{\mathrm{S,B}}^{-1}
[\mathsf S_\kappa]_{\mathrm{R,S}},
\label{eq:disc-left-etm}
\\
\mathsf P_{\mathrm{EtM}}^{\mathrm R}
&=
[\mathsf S_\kappa]_{\mathrm{R,S}}
[\mathsf G]_{\mathrm{R,R}^\star}^{-1}
\mathsf L_\varepsilon.
\label{eq:disc-right-etm}
\end{align}
The mixed Gram matrices $[\mathsf G]_{\mathrm{S,B}}$ and
$[\mathsf G]_{\mathrm{R,R}^\star}$ provide the discrete duality maps required
between the corresponding primal and dual spaces.

For comparison, we also consider the Calder\'on multiplicative
preconditioner (CMP)
\begin{equation}
\mathsf C_{\mathrm{CMP}}
=
[\mathsf S_{\kappa^\star}]_{\mathrm{B,S}}
[\mathsf G]_{\mathrm{S,B}}^{-1},
\end{equation}
which yields the preconditioned EFIE
\begin{equation}
\mathsf C_{\mathrm{CMP}}
[\mathsf S_\kappa]_{\mathrm{R,S}}
=
[\mathsf S_{\kappa^\star}]_{\mathrm{B,S}}
[\mathsf G]_{\mathrm{S,B}}^{-1}
[\mathsf S_\kappa]_{\mathrm{R,S}}.
\end{equation}
Here, $\kappa^\star$ denotes the wavenumber used in the regularizing
EFIE operator. In the numerical experiments below, we take either
$\kappa^\star=\kappa$ or $\kappa^\star=i\kappa$. More generally,
$\kappa^\star$ may be chosen as an appropriately damped complex
wavenumber (see~\cite{boubendir2014well}).

Figure~\ref{fig:gmres_sphere} compares the GMRES convergence of the
different formulations. Figure~\ref{fig:gmres_sphere}(a) examines the
effect of mesh refinement at the fixed wavenumber $\kappa=3\pi$. As
expected, due to the dense-discretization
breakdown of the EFIE, the iteration count of the
unpreconditioned formulation increases as the mesh is refined. By contrast, the
iteration counts of all the preconditioned formulations remain nearly
constant over the range of mesh sizes considered.

In Figure~\ref{fig:gmres_sphere}(b), the wavenumber is increased while
maintaining a fixed resolution of $8$ points per wavelength. As expected,
the number of iterations required by the unpreconditioned EFIE increases
rapidly with the electrical size of the problem. In contrast, both the
EtM- and MtE-preconditioned formulations substantially reduce the number
of GMRES iterations over the entire frequency range considered. Their
convergence is also comparable to that obtained with the Calderón
preconditioner.

\begin{figure}[H]
\centering
\includegraphics[width=\linewidth]
{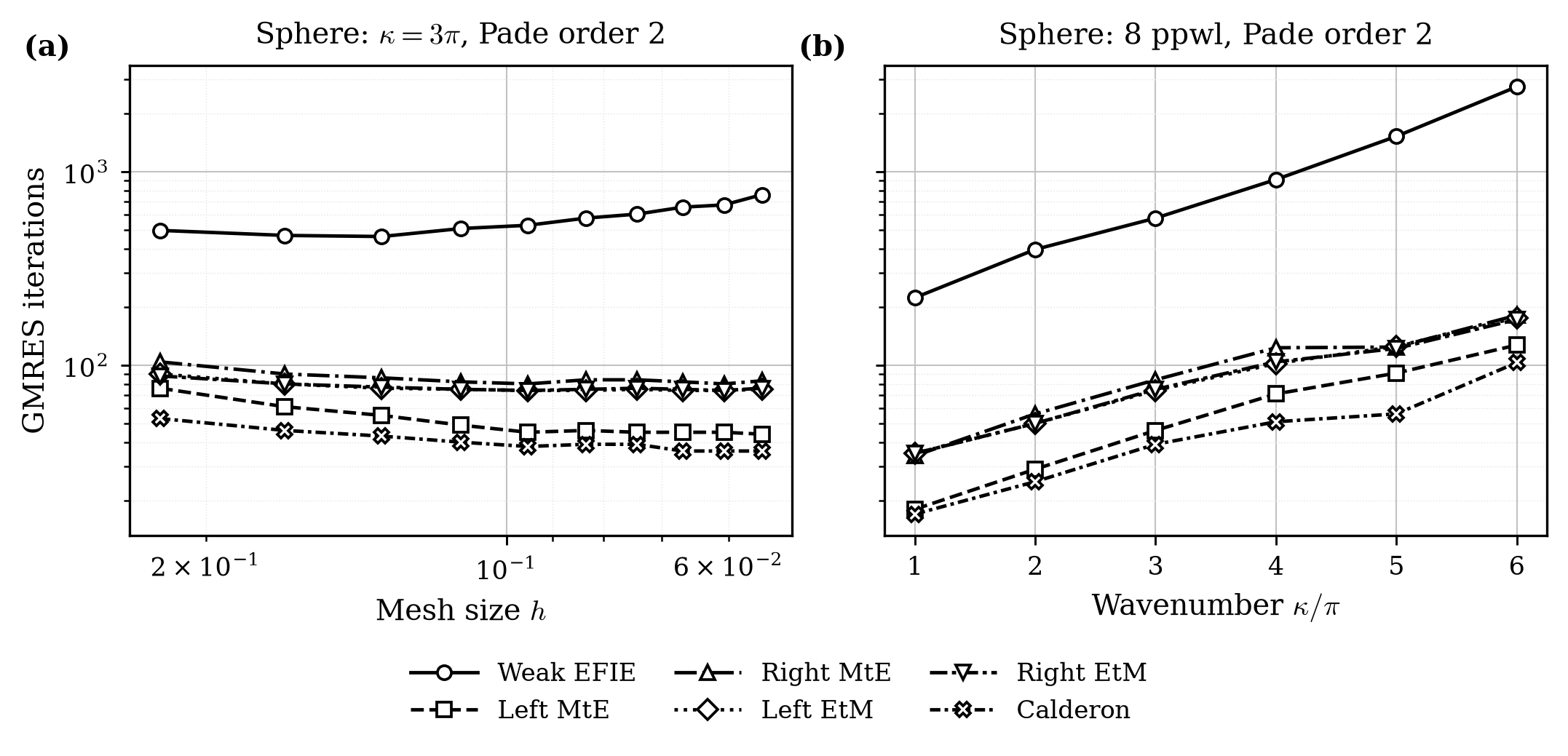}
\caption{GMRES iteration counts for the EFIE on the sphere using the
left- and right-preconditioned MtE and EtM formulations, the Calderón
preconditioner, and the unpreconditioned weak EFIE. Left: iteration counts
under mesh refinement for $\kappa=3\pi$ and Padé order $2$. Right:
iteration counts as a function of the wavenumber $\kappa$ for a fixed
resolution of $8$ points per wavelength and Padé order $2$.}
\label{fig:gmres_sphere}
\end{figure}

\begin{figure}[H]
\centering
\includegraphics[width=0.9\linewidth]{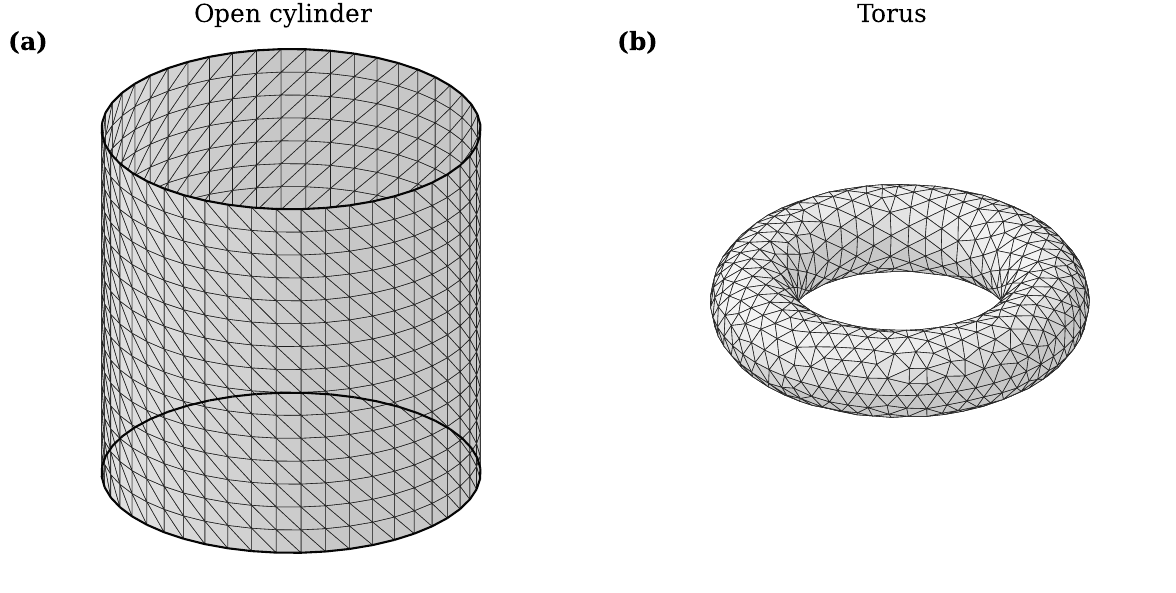}
\caption{(a) Open cylinder of radius $1$ and height $2$, without end caps. (b) Torus
with major radius $1$ and minor radius $0.3$.}
\label{fig:experiment_grids}
\end{figure}

We next consider an open cylinder without end caps (see Figure~\ref{fig:experiment_grids}, (a)). 
In spite the fact that this is not a closed boundary, Figure~\ref{fig:gmres_cylinder_no_caps}
shows that the OSRC preconditioners remain effective. Under mesh refinement, the growth in iteration count is
substantially reduced for all four OSRC formulations compared with the
unpreconditioned EFIE, with the MtE formulations displaying the most
stable behaviour over the range considered. A similar improvement is
observed as the wavenumber is increased at fixed resolution. These
results indicate that the effectiveness of the OSRC preconditioners is
not restricted, numerically, to smooth closed geometries.

\begin{figure}[H]
\centering
\includegraphics[width=\linewidth]
{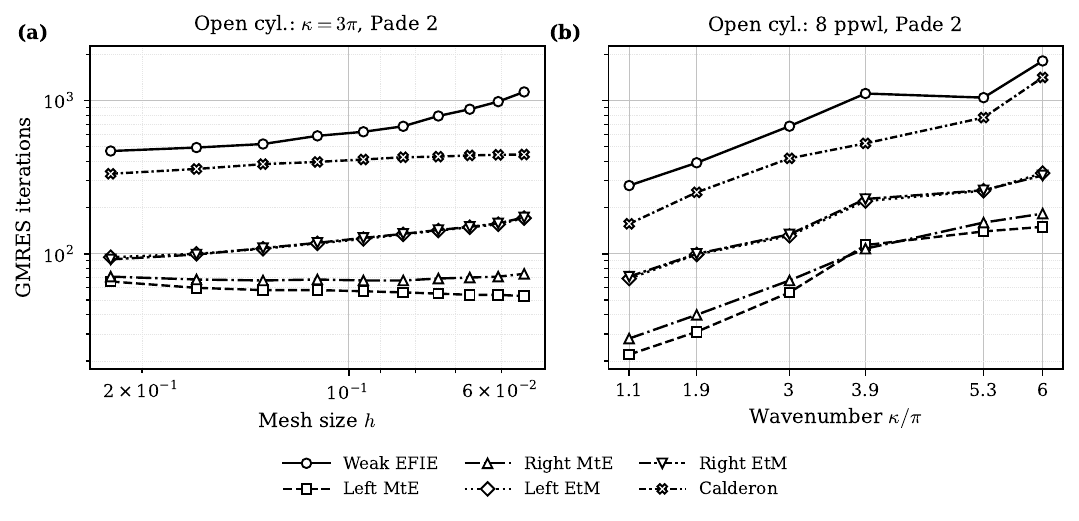}
\caption{GMRES iteration counts for the EFIE on the open cylinder without
end caps using the left- and right-preconditioned MtE and EtM formulations,
the Calderón preconditioner, and the unpreconditioned weak EFIE. Left:
iteration counts under mesh refinement for $\kappa=3\pi$ and Padé order
$2$. Right: iteration counts as a function of the wavenumber $\kappa$ for
a fixed resolution of $8$ points per wavelength and Padé order $2$. The
Calderón regularizing operator is evaluated at $\kappa^\star=i\kappa$.}
\label{fig:gmres_cylinder_no_caps}
\end{figure}

Finally, we consider a toroidal scatterer (see Figure~\ref{fig:experiment_grids}, (b)).
This geometry provides a numerical test of the formulations on a geometry for which the harmonic component of
the Hodge decomposition is present.

Figure~\ref{fig:gmres_torus} shows that all four OSRC formulations remain
effective under mesh refinement. While the iteration count of the
unpreconditioned EFIE increases substantially as the mesh is refined,
the OSRC-preconditioned formulations exhibit a much weaker dependence
on the mesh size. In particular, the left MtE formulation shows a
decrease in iteration count over the refinement sequence. The
Calder\'on formulation is also comparatively stable, although at a
higher iteration count in this experiment.

\begin{figure}[H]
\centering
\includegraphics[width=\linewidth]
{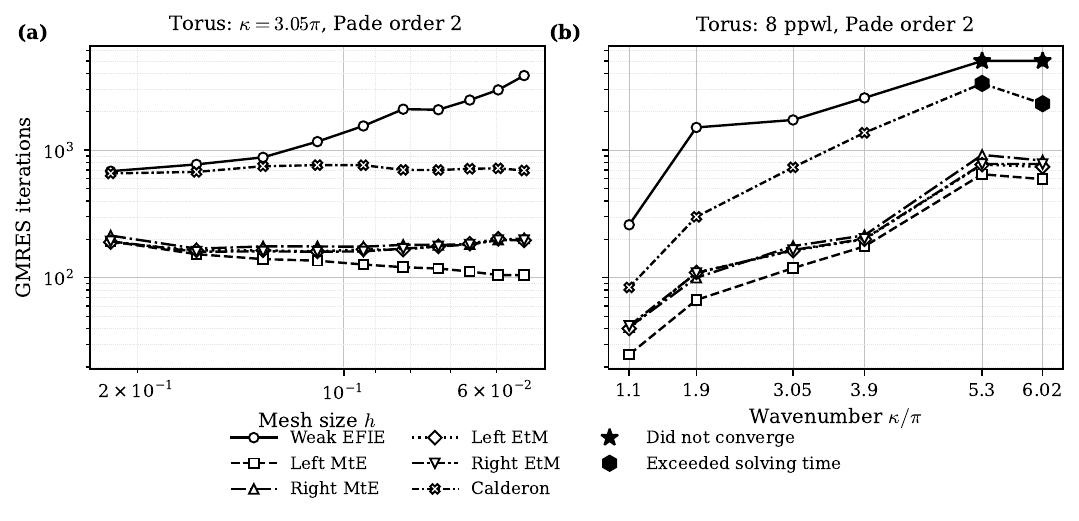}
\caption{GMRES iteration counts for the EFIE on the torus using the left-
and right-preconditioned MtE and EtM formulations, the Calderón
preconditioner, and the unpreconditioned weak EFIE. Left: iteration counts
under mesh refinement for $\kappa=3.05\pi$ and Padé order $2$. Right:
iteration counts as a function of the wavenumber $\kappa$ for a fixed
resolution of $8$ points per wavelength and Padé order $2$. The Calderón
regularizing operator is evaluated at $\kappa^\star=i\kappa$. Filled markers
at $\kappa=5.3\pi$ and $\kappa=6.02\pi$ identify runs that did not attain the
relative-residual tolerance $10^{-6}$: the weak EFIE reached the limit of
$5000$ inner iterations, while the Calderón runs were terminated after
exceeding the allocated solution time.}
\label{fig:gmres_torus}
\end{figure}

While the iteration counts provide a measure of the effectiveness of the
different preconditioners, they do not account for the computational cost
associated with their construction and application. We therefore examine
the setup and total computational times separately.

We first consider the setup cost of the different formulations.
Table~\ref{tab:setup_times} reports the assembly time of the weak EFIE in
seconds and uses it as a reference for the remaining formulations, whose
setup times are reported as percentages of the corresponding EFIE assembly
time. The physical EFIE is assembled densely, while the Calderón integral
operator is assembled using the Fast Multipole Method (FMM); see
\cite{greengard1987fast,chew2001fast}.

\begin{table}[ht]
\centering
\caption{Setup times for the experiments of
Figure~\ref{fig:gmres_sphere}(b). The weak EFIE assembly time is reported in
seconds, while the setup times of the preconditioned formulations are
reported as percentages of the corresponding weak EFIE assembly time.
The physical EFIE is assembled densely, while the Calder\'on EFIE at
$\kappa$ is assembled using FMM.}
\label{tab:setup_times}
\begingroup
\small
\setlength{\tabcolsep}{3.5pt}
\begin{tabular}{rrrrrrrr}
\toprule
$\kappa/\pi$ & DoFs & Weak EFIE [s]
& Left MtE & Right MtE & Left EtM & Right EtM & Calder\'on \\
\midrule
$1$ & 942   & 1.784    & 3.7\% & 4.2\% & 65.6\% & 65.4\% & 541.9\% \\
$2$ & 3183  & 19.926   & 1.8\% & 2.0\% & 46.5\% & 46.5\% & 128.5\% \\
$3$ & 6756  & 91.332   & 1.5\% & 1.6\% & 34.9\% & 34.9\% & 55.2\%  \\
$4$ & 12471 & 342.950  & 1.2\% & 1.2\% & 27.6\% & 27.6\% & 25.1\%  \\
$5$ & 18843 & 790.262  & 1.0\% & 1.0\% & 25.8\% & 25.8\% & 16.7\%  \\
$6$ & 26463 & 1718.658 & 0.8\% & 0.8\% & 23.4\% & 23.4\% & 11.1\%  \\
\bottomrule
\end{tabular}
\endgroup
\end{table}

Table~\ref{tab:setup_times} highlights the different construction costs of
the preconditioners. The MtE formulations have a substantially lower setup
cost than the EtM formulations. The latter require the construction of
BC/RBC basis functions on a barycentrically refined mesh, leading to a
considerably larger discrete problem during the construction of the
preconditioner. In contrast, the MtE formulations are constructed directly
on the primal mesh and avoid this additional discretization. The Calderón
formulation also requires a barycentrically refined discretization, although
its integral operators benefit from FMM-accelerated assembly.

We next consider the overall computational cost. For each preconditioned
formulation, the total time includes both the setup of the preconditioner
and the GMRES solution time. For the unpreconditioned weak EFIE, no
preconditioner is constructed, and hence only the GMRES solution time is
included. Table~\ref{tab:computational_times} reports the absolute weak
EFIE solution time together with the total time of each preconditioned
formulation, expressed as a percentage of the corresponding weak EFIE
solution time.

\begin{table}[ht]
\centering
\caption{Total computational times for the experiments of
Figure~\ref{fig:gmres_sphere}(b). The weak EFIE solution time is reported in
seconds, while the total times of the preconditioned formulations,
including both preconditioner setup and GMRES solution, are reported as
percentages of the corresponding weak EFIE solution time.}
\label{tab:computational_times}
\begingroup
\small
\setlength{\tabcolsep}{3.5pt}
\begin{tabular}{rrrrrrrr}
\toprule
$\kappa/\pi$ & DoFs & Weak EFIE [s]
& Left MtE & Right MtE & Left EtM & Right EtM & Calder\'on \\
\midrule
$1$ & 942   & 0.111   & 90.5\% & 130.8\% & 1620.9\% & 1629.7\% & 51228.9\% \\
$2$ & 3183  & 0.964   & 75.4\% & 116.1\% & 1406.3\% & 1407.9\% & 15480.1\% \\
$3$ & 6756  & 8.839   & 36.4\% & 54.6\%  & 508.7\%  & 509.9\%  & 3906.2\%  \\
$4$ & 12471 & 49.572  & 23.4\% & 34.9\%  & 271.5\%  & 271.3\%  & 1596.5\%  \\
$5$ & 18843 & 186.639 & 15.9\% & 18.1\%  & 151.2\%  & 150.2\%  & 601.5\%   \\
$6$ & 26463 & 657.455 & 10.2\% & 11.9\%  & 86.8\%   & 86.6\%   & 420.8\%   \\
\bottomrule
\end{tabular}
\endgroup
\end{table}

Table~\ref{tab:computational_times} shows that the reduction in GMRES
iterations translates into a substantial reduction in total computational
time once the problem is sufficiently large. The MtE formulations exhibit
the lowest cost: the left-preconditioned MtE decreases from
$90.5\%$ of the weak-EFIE solution time at $\kappa=\pi$ to $10.2\%$ at
$\kappa=6\pi$. For the smaller problems, the setup overhead can outweigh
the shorter iterative solve, particularly for the EtM and Calder\'on
formulations.

A comparison of Tables~\ref{tab:setup_times}
and~\ref{tab:computational_times} also shows that the overall savings are
not determined by setup time alone. Although the relative setup cost of
the MtE preconditioner decreases with the problem size, its contribution
remains small compared with the time required by the unpreconditioned EFIE
during the iterative solution. The main computational saving therefore
arises from the reduction in GMRES iterations which can turn expensive when 
resorting to matrix-free methods like FMM.

The EtM and Calderón formulations reduce the iteration count, but require a
more expensive construction involving dual basis functions on a
barycentrically refined mesh. Moreover, each Calderón iteration applies an
integral operator on that refined mesh, which dominates its total time in
these experiments. The MtE formulations combine comparable iterative
convergence with a refinement-free construction, resulting in the lowest
total computational cost among the formulations considered here.

\section{Conclusions}\label{sect:conclusions}
In this work, we have presented a continuous and discrete analysis of
OSRC-based preconditioning strategies for the electric field integral
equation, considering both Electric-to-Magnetic (EtM) and
Magnetic-to-Electric (MtE) approximations.

At the continuous level, the four resulting compositions were analyzed
by means of pseudodifferential properties of the EFIE and OSRC
operators. The analysis shows that the OSRC approximations of the EtM
and MtE maps reflect the local structure of the EFIE and its inverse,
respectively, and that the resulting formulations are well-posed.

The continuous equivalence of these formulations does not, however,
translate directly to their Galerkin discretizations. We have shown
that both left and right MtE preconditioning are compatible with the
standard RWG/SNC discretization and can therefore be implemented
without barycentric refinement. In contrast, the corresponding EtM
compositions require an appropriate discrete duality map, leading to
the use of BC/RBC spaces on a barycentrically refined mesh. This
distinction highlights the role of the discrete trial and test spaces
when transferring operator identities from the continuous to the
Galerkin setting.

The numerical experiments support our theoretical and discrete analyses. 
On the sphere, the spectral results illustrate the importance of an appropriate 
discretization of the MtE and EtM operators, while both OSRC strategies 
substantially reduce the growth in GMRES iterations observed for the 
unpreconditioned EFIE under mesh refinement and increasing electrical size. 
The torus experiments show a similar behaviour in the presence of a 
nontrivial harmonic component, and the open-cylinder results provide numerical 
evidence that the effectiveness of the preconditioners can extend beyond the 
smooth closed-surface setting considered in the continuous analysis. The computational 
experiments further expose a significant practical distinction between the formulations: 
although both EtM and MtE provide effective iterative preconditioning, the refinement-free 
MtE formulations have a substantially lower setup cost and, for the larger problems considered, 
the lowest total computational cost. Thus, the choice between EtM and MtE approximations depends 
not only on their continuous operator properties but also on how their discreizations interact 
with the underlying boundary element spaces.

\section*{Acknowledgements}
I am grateful to Marion Darbas for her contributions to an earlier version
of this work. Several arguments used in the proofs presented here build on
ideas developed in that earlier work.

\section{Appendix}\label{sect:appendix}

\subsection{Approximation of \texorpdfstring{$(1+z)^{\half}$}{sqrt(1+z)}}\label{sect:app1}

We first consider the scalar approximation of the function
\[
(1+z)^{\half},
\qquad z\in\mathbb C.
\]
Following \cite{milinazzo1997rational}, we introduce a rotation of the
branch cut by an angle $\theta_1$ and write
\begin{equation}
\label{eq:pade_app}
(1+z)^{\half} \approx e^{i\theta_\half} \mathrm{R}_{\mathrm{L}_{1}}(e^{-i\theta_1}(1+z)-1) = \mathrm{C}_{0} + \sum_{\ell=1}^{\mathrm{L}_{1}} \dfrac{\mathrm{A}_{\ell}  z}{1+ \mathrm{B}_{\ell}z}  = \mathrm{F}_{0}-\sum_{\ell=1}^{\mathrm{L}_{1}} \dfrac{\mathrm{A}_{\ell}  }{\mathrm{B}_{\ell}(1+ \mathrm{B}_{\ell}z)}.
\end{equation}

Here $\mathrm{R}_{\mathrm{L}_{1}}$ denotes the standard real-valued Pad\'e approximation of
order $\mathrm{L}_{1}$:
\[
(1+z)^{\half} \approx \mathrm{R}_{\mathrm{L}_{1}}(z)=1+\sum_{\ell=1}^{\mathrm{L}_{1}}\dfrac{a_{\ell}z}{1+b_{\ell}z},
\]
with coefficients
\[
a_{\ell}
=
\frac{2}{2\mathrm{L}_{1}+1}
\sin^2\left(\frac{\ell\pi}{2\mathrm{L}_{1}+1}\right),
\qquad
b_{\ell}
=
\cos^2\left(\frac{\ell\pi}{2\mathrm{L}_{1}+1}\right),
\qquad
1\leq \ell\leq \mathrm{L}_{1}.
\]

The complex coefficients resulting from the branch-cut rotation are given by
\[
\mathrm{A}_{\ell}
=
\frac{
e^{-i\theta_\half}a_{\ell}
}{
\left(
1+b_{\ell}(e^{-i\theta_1}-1)
\right)^2
},
\qquad
\mathrm{B}_{\ell}
=
\frac{
-e^{-i\theta_1}b_{\ell}
}{
1+b_{\ell}(e^{-i\theta_1}-1)
},
\]
for $1\leq \ell\leq \mathrm{L}_{1}$, while
\[
\mathrm{C}_{0}
=
e^{i\theta_\half}\mathrm{R}_{\mathrm{L}_{1}}(e^{-i\theta_1}-1),
\qquad
\mathrm{F}_{0}
=
\mathrm{C}_{0}+\sum_{\ell=1}^{\mathrm{L}_{1}}\frac{\mathrm{A}_\ell}{\mathrm{B}_\ell}.
\]

The branch-cut rotation plays an essential role in the approximation of
evanescent modes.
Indeed, the singular set of the square-root function corresponds to the
negative real axis, which is precisely where the evanescent spectrum of
the surface Laplacian is located.
Rotating the branch cut moves the singularity away from the relevant
spectral region and significantly improves the stability and accuracy of
the rational approximation
\cite{milinazzo1997rational,AntoineDarbasLuCMAME}.

\subsection{Approximation of \texorpdfstring{$(1+z)^{-\half}$}{1/sqrt(1+z)}}\label{sect:app2}

A stable approximation of the inverse square-root function is also
required.
A classical approach is based on continued fractions
\cite{cockburn1988maxwell}.
Using fixed-point and residue techniques,
\cite{CDLLJCP2017} derives the rational approximation
\begin{equation}
\label{Isqrt_real}
(1+z)^{-\half} \approx  \mathrm{R}^{inv}_{\mathrm{L}_{2}}(z)=\sum_{\ell=0}^{\mathrm{L}_{2}-1}\dfrac{\mathrm{C}_{\ell}}{\mathrm{D}_{\ell}+z},
\end{equation}
with
\[
\mathrm{C}_{\ell}
=
\frac{
1+\tan^2\left(
\frac{\pi}{2\mathrm{L}_{2}}\left(\frac12+\ell\right)
\right)
}{\mathrm{L}_{2}},
\]
and
\[
\mathrm{D}_{\ell}
=
1+
\tan^2\left(
\frac{\pi}{2\mathrm{L}_{2}}\left(\frac12+\ell\right)
\right),
\qquad
0\leq \ell\leq \mathrm{L}_{2}-1.
\]

These coefficients correspond to a Pad\'e approximation whose branch cut
lies along the negative real axis.
As before, we rotate the branch cut by an angle $\theta_2$:
\[
(1+z)^{-\half}
=
e^{i\theta_2/2}
\left(
e^{i\theta_2}(1+z)
\right)^{-\half}
=
e^{i\theta_2/2}
\left(
1+[e^{i\theta_2}(1+z)-1]
\right)^{-\half}.
\]

Using \eqref{Isqrt_real}, we obtain
\begin{equation}
\label{Isqrt}
(1+z)^{-\half}  \approx e^{i\theta_2/2}  \mathrm{R}^{inv}_{\mathrm{L}_{2}}(e^{i\theta_2} (1+z)-1) =  \sum_{\ell=0}^{\mathrm{L}_{2}-1} \dfrac{\mathrm{R}_{\ell}}{\mathrm{S}_{\ell}+z},
\end{equation}
where
\[
\mathrm{R}_{\ell}
=
e^{i \theta_2/2} \mathrm{C}_{\ell},
\qquad
\mathrm{S}_{\ell}
=
1-e^{i \theta_2}+\mathrm{D}_{\ell} e^{i \theta_2},
\qquad
0 \leq \ell \leq \mathrm{L}_{2}-1.
\]

\bibliography{bibliography}{}
\bibliographystyle{plain}
\end{document}